\documentclass{article}
\usepackage{amsmath,amssymb,amsfonts,amsthm}
\usepackage[margin=1in]{geometry}
\usepackage{aliascnt}
\usepackage{hyperref}
\usepackage{enumitem}
\usepackage{float}
\usepackage{placeins}
\usepackage{mathtools}
\usepackage{todonotes}
\usepackage{tikz-cd}

\usepackage{graphicx}
\theoremstyle{definition}
\newtheorem{theorem}{Theorem}[section]
\newtheorem{lemma}[theorem]{Lemma}
\newtheorem{definition}[theorem]{Definition}
\newtheorem{proposition}[theorem]{Proposition}
\newtheorem{corollary}[theorem]{Corollary}
\newtheorem{example}[theorem]{Example}

\newaliascnt{remark}{theorem}
\newtheorem*{theorem*}{Theorem}
\newtheorem{remark}[remark]{Remark}
\aliascntresetthe{remark}

\DeclareMathOperator{\Spin}{Spin}
  
\hypersetup{colorlinks,breaklinks, linkcolor = purple, citecolor = teal, urlcolor = purple, }
\setlist{nosep,leftmargin=*}
\numberwithin{equation}{section}

\title{Second-Jet Equivariant $\eta$ Separations on Lens Spaces}
\author{Sanchita Sharma\thanks{Email address: \href{mailto:sharma.sanchita019@gmail.com}{sharma.sanchita019@gmail.com}}}
\date{}

\begin{document}

\maketitle

\begin{abstract}
Lens spaces are useful test examples in spectral geometry because their spin Dirac eigenspaces admit explicit congruence descriptions.  We use these descriptions to study equivariant $\eta$ invariants for three-dimensional lens spaces with the round metric and the standard coordinate-torus action, retaining the spin-Fourier character of each eigenspace rather than only the ordinary scalar $\eta$ value.  For the square family $L(\ell^2,\ell-1)$ and $L(\ell^2,2\ell-1)$, with $\ell\geq 5$ odd, we obtain a residual-circle equivariant $\eta$ separation: the ordinary $\eta$ values agree, and the first derivative of the residual $\eta$ germ vanishes by symmetry, but the second derivative is nonzero.  For $L(25,4)$ versus $L(25,9)$, the normalized second derivative is $-6080$.  Thus, the residual-circle equivariant $\eta$ germ detects a distinction invisible to the ordinary $\eta$ invariant.  The calculation uses spin-Fourier residues directly; perturbative Hessian signs serve only as motivation and are not part of the invariant.
\end{abstract}

\tableofcontents

\section{Introduction}

Lens spaces form one of the standard families of spherical space forms and have long served as test examples in topology, spectral geometry, and global analysis.  We write
\begin{equation}
L(p,q)=S^3/\mathbb Z_p,
\qquad
\gamma(z_1,z_2)=(\omega z_1,\omega^q z_2),
\qquad
\omega=\exp(2\pi i/p),
\end{equation}
where $q$ is coprime to $p$.  Thus $p$ is the order of the cyclic quotient, while $q$ specifies how the generator acts on the two complex coordinates of $S^3\subset \mathbb C^2$.  The topology of lens spaces is classical; they appear already in the basic study of covering spaces and Reidemeister-Whitehead torsion, and they remain useful because many refined invariants can be computed explicitly on them~\cite{HatcherAlgebraicTopology,MilnorWhiteheadTorsion}. From the analytic side, lens spaces also provide concrete examples for comparing torsion, spectrum, and representation-theoretic refinements~\cite{RayLensTorsion,GilkeySpherical}.  Recent work on lens-space spectra and Dirac multiplicities gives a modern spectral-geometric context for these explicit calculations~\cite{BoldtDiracLensPublished,BoldtLauretDiracMultiplicities,LauretMiatelloRossettiRecentLensSpectra}. 

The Dirac spectrum of spherical space forms, including lens spaces, is classical.  We use the representation-theoretic description of spin Dirac eigenspaces due to B\"ar and Gilkey, together with the standard spin-geometric conventions of Friedrich and Lawson-Michelsohn~\cite{GilkeySpherical, BaerSpaceForms,FriedrichDirac,LawsonMichelsohn}.  In this description, the spinor eigenspaces are described by congruence conditions on spin weights.  For three-dimensional lens spaces with odd $p$, these weights may be represented by pairs of odd integers.  The usual scalar multiplicity only counts how many weights occur.  The equivariant viewpoint keeps track of how symmetries act on the corresponding eigenspaces.

$\eta$ invariants measure spectral asymmetry.  The non-equivariant $\eta$ invariant of Atiyah-Patodi-Singer takes a signed sum over the spectrum and produces a scalar spectral invariant~\cite{APS1,APS2,APS3}.  Donnelly's equivariant $\eta$ invariant refines this construction by inserting the action of a symmetry group on each eigenspace~\cite{DonnellyGSpaces}.  
Instead of retaining only the scalar asymmetry between positive and negative eigenvalues, one can retain character data.  For lens spaces, this refinement is especially concrete: the eigenspaces are described explicitly, and the group action is visible at the level of spin weights.  The scalar $\eta$ values of lens spaces are closely related to classical finite trigonometric sums and Dedekind sums; coincidences among such scalar values have been studied, for example, by Katase and in the broader Dedekind-sum literature~\cite{KataseDedekindEta,KataseClassifyingEta,KirbyMelvinDedekind,JabukaRobinsWangDedekind}.

The coordinate torus acting on $S^3\subset\mathbb C^2$ gives the equivariant structure used below.  It rotates the two complex coordinates independently.  On spinors, these rotations lift to spinorial characters.  A one-variable character sees only the weight coming from one coordinate circle, while a two-variable character retains both coordinate weights.  Passing from one variable to two variables is formally small but conceptually important: a one-variable specialization can lose information that remains visible in the full torus-equivariant $\eta$ character.

The classical lens-space Dirac spectrum and the general equivariant $\eta$ formalism are known, but explicit examples are still useful when the scalar $\eta$ value vanishes and the first nonzero residual contribution occurs only at second order.  The finite-shell characters already show that equivariant weights can differ.  The eta-theoretic question is sharper: can this weight difference be assembled into a residual-circle $\eta$ germ whose first nonzero term is quadratic?

The answer is a residual-circle second-jet separation.  A residual circle is a one-parameter circle inside the coordinate torus.  Restricting the torus-equivariant $\eta$ character to such a circle gives a one-variable $\eta$ germ near the identity element.  Its second jet is the second derivative of this germ at the identity.  For the residual germs considered below, the first derivative at the identity vanishes by the symmetry of the finite sum.  The nontrivial point is that the zeroth-order $\eta$ difference also vanishes for the square family, while the second derivative is nonzero.

The first explicit pair is $L(25,4)$ and $L(25,9)$.  Recent APS-rho literature also uses explicit lens-space rho formulae and Dedekind-Rademacher sums~\cite{ToffoliAPSRhoLinks}.  For this pair, the ordinary scalar $\eta$ value cancels, but the coordinate-equivariant spin-Fourier residue does not.  After passing to the residual-circle equivariant $\eta$ germ, the first residual derivative vanishes by symmetry and the second residual derivative is equal to $-6080$ in the normalized convention used below.  We then extend the computation to the square family
\begin{equation}
L(\ell^2,\ell-1)
\qquad\text{and}\qquad
L(\ell^2,2\ell-1),
\end{equation}
where $\ell\geq 5$ is odd.  The ordinary $\eta$ equality for this family follows from the standard finite trigonometric formula for lens-space $\eta$ invariants; we include a direct cotangent-sum proof because the same normalization and pairing conventions are used later in the residual-circle calculation.

\begin{theorem}[Main theorem: residual-circle second-jet separation]
\label{thm:main-residual-circle-second-jet}
Let $\ell\geq 5$ be odd.  Let
\begin{equation}
P=\ell^2,
\qquad
q_1=\ell-1,
\qquad
q_2=2\ell-1,
\end{equation}
and define the bare lens spaces
\begin{equation}
M_{\ell,1}=L(P,q_1),
\qquad
M_{\ell,2}=L(P,q_2).
\end{equation}
For the residual-circle specialization of the $T^2$-equivariant $\eta$ character considered in \autoref{sec:T2-equivariant-eta}, the relative $\eta$ germ for $M_{\ell,1}$ and $M_{\ell,2}$ has zero value at the identity.  Its first residual derivative vanishes by symmetry, but its second residual derivative is nonzero.  In the normalized finite-sum convention defined in \autoref{subsec:residual-circle-normalization}, the second derivative is
\begin{equation}
C_\ell
=
-\frac{\ell^2(\ell-1)^2(\ell+1)(11\ell^2+20\ell+81)}{180}.
\end{equation}
This normalized second derivative is nonzero for every odd $\ell\geq 5$.  For $\ell=5$, the pair is $L(25,4)$ and $L(25,9)$, and the normalized second derivative is $C_5=-6080$.  Thus the residual-circle equivariant $\eta$ germ detects a second-order separation: the ordinary $\eta$ value vanishes, the first derivative is forced to vanish by symmetry, and the second derivative is nonzero.
\end{theorem}

The normalization factor relating this finite sum to a Donnelly-APS convention is fixed in~\autoref{subsec:residual-circle-normalization}; it is nonzero and depends only on the global convention and on $P$.  Hence the vanishing and nonvanishing assertions are normalization-independent.  The rho statement is numerical for the fixed residual circle, not an analytic-surgery structure-class statement.

The theorem is an equivariant spin statement.  The chosen coordinate torus action is part of the data.  We do not claim a classification theorem for lens spaces, nor do we claim that the same pairs are invisible to all finite-cyclic APS rho invariants.  For the chosen coordinate-equivariant spin data, the scalar $\eta$ value does not see the difference; the first residual derivative vanishes for structural reasons; the second residual derivative detects the separation.

A perturbative Hessian calculation is included only as motivation.  It explains why the second coordinate spin weight should not be discarded.  The Hessian signs themselves depend on a perturbation chamber, so they are not used as invariants.  The invariant calculation instead retains both coordinate spin weights directly and assembles them into the torus-equivariant $\eta$ character.

The representation-ring and equivariant $K$-homology language is used only for positioning.  It indicates where the calculation sits relative to equivariant index theory and secondary invariants, following Atiyah-Segal and Kasparov~\cite{AtiyahSegal,KasparovKK}.  We do not construct a Higson-Roe analytic-surgery class.

The organization is as follows.  \autoref{sec:spin-fourier-residues} extracts one- and two-variable spin-Fourier residues from the known lens-space Dirac spectrum.  \autoref{sec:perturbative-diagnostic-short} gives the perturbative diagnostic and explains why Hessian signs are not used as invariants.  \autoref{sec:stage5-scalar-eta-square} separates the spin-Fourier residue from scalar shadows, including the ordinary $\eta$ shadow.  \autoref{sec:T2-equivariant-eta} proves the eta-theoretic interpretation, the explicit $L(25,4)$ versus $L(25,9)$ second-jet separation, and the infinite square-family residual-circle rho-number formula.  \autoref{sec:scope-limitations} states the scope relative to classical rho invariants and analytic surgery. \autoref{app:first-useful-spin-fourier-table} provides the full first-useful spin-Fourier residue table used as a bookkeeping check.  \autoref{app:L25-eta-second-derivative} gives an exact verification of the $L(25,4)$ versus $L(25,9)$ second derivative. \autoref{app:block-moment-verification} provides a symbolic verification of the block moment identities used in the square-family coefficient calculation.

\section{Spin-Fourier residues of lens-space Dirac eigenspaces} \label{sec:spin-fourier-residues}

\subsection{Spin double-cover target}
\label{subsec:stage-1A-spin-cover-target}

We first fix the equivariant character ring in which the lens-space Dirac eigenspaces will be placed. In the lower-dimensional equivariant models that motivate this calculation, the secondary representation data are often written using an ordinary circle variable. For spin Dirac operators on lens spaces, however, the natural target is slightly finer. The coordinate circle action on the underlying manifold lifts to the spinor bundle through the spin double cover, so the corresponding eigenspace characters are naturally characters of $\operatorname{Spin}(2)$, not only of $SO(2)$.

Let $L(p,q)=S^3/\mathbb Z_p$, with generator $\gamma(z_1,z_2)=(\omega z_1,\omega^q z_2)$ and $\omega=\exp(2\pi i/p)$. Here $p\geq 3$ is odd unless stated otherwise, and $q$ is coprime to $p$. We use the coordinate circle action $e^{i\theta}\cdot [z_1,z_2]=[e^{i\theta}z_1,z_2]$. This is an $SO(2)$-action on the lens space. After passing to spinors, we keep the lifted action on the spinor bundle, whose weights lie naturally in the character group of $\operatorname{Spin}(2)$.

Let $u$ denote the fundamental character of $\operatorname{Spin}(2)$.  A spin weight will be written as $\mu=\frac12(a_1,a_2), a_1,a_2\in 2\mathbb Z+1.$ The coordinate action rotating $z_1$ selects the first spin weight.  Thus the corresponding spin-Fourier monomial is $\mu=\frac12(a_1,a_2) \longmapsto  u^{a_1}.$ The appearance of odd exponents reflects the spinorial nature of the lifted circle action.  For this reason, forcing the calculation prematurely into the ordinary $SO(2)$-character variable would obscure the natural representation-theoretic structure.

\begin{definition}[Spin-Fourier character]
Let $E$ be a finite-dimensional $\operatorname{Spin}(2)$-representation arising from a Dirac eigenspace on $L(p,q)$.  Suppose that $E$ decomposes into spin weights as $E=\bigoplus_{a\in 2\mathbb Z+1} E_a.$ Its spin-Fourier character is defined by $\chi_{\operatorname{Spin}}(E;u)=\sum_{a\in 2\mathbb Z+1} \dim(E_a)u^a.$ \end{definition}

The reflection symmetry, when present, pairs opposite spin weights.  Thus, the secondary Fourier sector is naturally generated by the symmetric combinations $A_n(u)=u^n+u^{-n}, n\geq 1.$ In the lens-space examples below, the relevant exponents $n$ are typically odd, because the spin weights are half-integral before being written in the $u$-variable.

\begin{definition}[Reduced spin-Fourier target]
The reduced spin-Fourier target is the quotient of the reflection-invariant spin-character ring by the trivial character.  We denote it by $\widetilde R_{\operatorname{spin}} = R(\operatorname{Spin}(2))^{\mathbb Z/2}/\mathbb Z\mathbf 1.$ Equivalently, it is the free abelian group generated by the nontrivial symmetric modes: $\widetilde R_{\operatorname{spin}} = \bigoplus_{n\geq 1}\mathbb Z\langle A_n(u)\rangle, A_n(u)=u^n+u^{-n}.$ \end{definition}

The reduction by the trivial character is the spinorial analogue of the reduced secondary target used in the $O(2)$-equivariant formalism.  It removes the scalar dimension part and keeps only the nontrivial spin-Fourier information. The examples below are designed to detect cases where ordinary dimensions agree but spin-Fourier characters differ.

Let $E_k^+(p,q)$ and $E_k^-(p,q)$ denote the eigenspaces of the Dirac operator on $L(p,q)$ corresponding to the eigenvalues $\lambda_k^\pm=\pm\left(k+\frac32\right), k\geq 0.$ Their spin-Fourier characters will be denoted by $\chi_k^+(p,q;u)=\chi_{\operatorname{Spin}}(E_k^+(p,q);u)$ and $\chi_k^-(p,q;u)=\chi_{\operatorname{Spin}}(E_k^-(p,q);u)$.  The basic virtual spectral difference at level $k$ is $\Delta_k(p,q;u)=\chi_k^+(p,q;u)-\chi_k^-(p,q;u)$. Its reduced class is $\widetilde\Delta_k(p,q;u)=[\Delta_k(p,q;u)]_{\operatorname{red}}\in\widetilde R_{\operatorname{spin}}$.  This is an equivariant spin-spectral quantity: the chosen lifted coordinate circle action determines which spin weight is seen by $u$.

The reduced spin-Fourier character can see information invisible to ordinary dimension.  For example, two eigenspaces may have the same dimension, $\dim E_k^+(p,q)=\dim E_\ell^-(p,q),$ but different spin-Fourier characters: $\chi_k^+(p,q;u)\neq \chi_\ell^-(p,q;u).$ In that case the reduced difference $\chi_k^+(p,q;u)-\chi_\ell^-(p,q;u)$ can be nonzero even though the ordinary dimension difference vanishes.

This is the basic mechanism used in the lens-space computation.  The known Dirac spectrum supplies the eigenspaces.  The spin-Fourier refinement keeps the weight data of those eigenspaces under the lifted coordinate circle action.  The reduced character then extracts a secondary representation-valued residue.

\begin{example}[The form of the first lens-space residues]
In the examples below, one obtains reduced differences of the form $A_m(u)-A_n(u), A_j(u)=u^j+u^{-j}.$ For instance, the first nontrivial $L(5,2)$ calculation produces $A_1(u)-A_3(u) = (u+u^{-1})-(u^3+u^{-3}).$ This difference has zero ordinary dimension after evaluation at $u=1$: $\left(A_1(1)-A_3(1)\right)=2-2=0.$ Nevertheless, it is nonzero in the reduced spin-Fourier target: $A_1(u)-A_3(u)\neq 0 \ \text{in} \ \widetilde R_{\operatorname{spin}}.$ \end{example}

\begin{proposition}[Spin-cover target for the lens-space Dirac computation]
The natural equivariant target for the lens-space Dirac eigenspace calculation is the reduced spin-Fourier target $\widetilde R_{\operatorname{spin}}$, rather than the ordinary reduced $SO(2)$-target.  In this target, the level-wise reduced spectral difference $\widetilde\Delta_k(p,q;u)=[\chi_k^+(p,q;u)-\chi_k^-(p,q;u)]_{\operatorname{red}}$ retains the nontrivial spinorial Fourier information of the Dirac eigenspaces.
\end{proposition}

\begin{proof}
The coordinate circle action on $L(p,q)$ lifts naturally to the spinor bundle through the spin double cover.  Consequently, the weights of the lifted action are most naturally expressed through the $\operatorname{Spin}(2)$-character variable $u$.  A spin weight $\mu=\frac12(a_1,a_2)$ contributes $u^{a_1}$, and since $a_1$ is odd in the spin lattice, the resulting characters naturally contain odd powers of $u$.

The ordinary dimension of a character is obtained by evaluating at $u=1$.  Therefore the quotient by the trivial character removes the scalar dimension part and retains precisely the nontrivial spin-Fourier modes.  The expression $\chi_k^+(p,q;u)-\chi_k^-(p,q;u)$ therefore measures the difference between the positive and negative Dirac eigenspaces at the level of spin-Fourier characters.  Passing to the quotient by $\mathbb Z\mathbf 1$ gives a well-defined reduced secondary quantity in $\widetilde R_{\operatorname{spin}}$.
\end{proof}

This spin-cover refinement is the first step in replacing the earlier toy scalar congruence by the actual Dirac-spinor calculation.  The next step is to compute the characters $\chi_k^\pm(p,q;u)$ using the affine congruence-lattice formula for the Dirac spectrum on lens spaces.

\subsection{Affine congruence lattice and parity-shell formula}
\label{subsec:stage-1B-affine-parity-shell}

We replace the scalar congruence used in the preliminary examples by the actual spinor congruence formula for Dirac eigenspaces on lens spaces.  The Dirac spectrum of lens spaces is known, and we use the known affine congruence-lattice description of the spin Dirac spectrum and retain the $\operatorname{Spin}(2)$-Fourier weight information which will later feed into the reduced residue.

For odd $p$, $L(p,q)$ has a unique spin structure.  The even-$p$ case requires keeping track of the choice of spin structure and is not needed for the first lens-space computation.

Recall that $L(p,q)=S^3/\mathbb Z_p, \gamma(z_1,z_2)=(\omega z_1,\omega^qz_2), \omega=\exp(2\pi i/p).$ The spin weights are written as $\mu=\frac12(a_1,a_2), a_1,a_2\in 2\mathbb Z+1.$ The lens-space descent condition is the spinor congruence $a_1+q a_2\equiv 0\pmod p.$ This congruence is the spinorial analogue of the scalar monomial congruence, but it is not the same object.  The oddness of $a_1$ and $a_2$ reflects the half-integral spin weights.

\begin{definition}[Affine congruence lattice]
For odd $p$ and $q\in(\mathbb Z/p\mathbb Z)^\times$, define the affine spin congruence lattice by
\begin{equation}
\mathcal L(p,q)=\left\{(a_1,a_2)\in(2\mathbb Z+1)^2:\ a_1+q a_2\equiv0\pmod p\right\}.
\end{equation}
A point $a=(a_1,a_2)\in\mathcal L(p,q)$ contributes the spin-Fourier monomial $a=(a_1,a_2)\longmapsto u^{a_1}.$ \end{definition}

The ordinary Dirac multiplicity is obtained by counting lattice points.  For the present purpose we do not only count them; we keep track of the first coordinate $a_1$, because $a_1$ is the spin-Fourier weight for the lifted circle action rotating $z_1$.

Two elementary functions on lattice points enter the Dirac formula.

\begin{definition}[One-norm and sign parity]
For $a=(a_1,a_2)\in(2\mathbb Z+1)^2$, define the spin one-norm by $\|a\|_1=\frac{|a_1|+|a_2|}{2}.$ Define the sign-counting parity by $R(a)=\#\{j\in\{1,2\}:a_j<0\}.$ Thus $R(a)$ is the number of negative entries among $a_1$ and $a_2$.
\end{definition}

The eigenvalues of the round Dirac operator on a three-dimensional lens space are $\lambda_k^\pm=\pm\left(k+\frac32\right), k\geq0.$ The eigenspaces for $\lambda_k^+$ and $\lambda_k^-$ are not obtained from only the outer shell $\|a\|_1=k+1$.  The correct formula involves a finite sum over lower shells.  This is why the naive scalar-shell calculation is not correct.

For each $k\geq0$ and $0\leq r\leq k$, define the shell
\begin{equation}
\mathcal L_{k,r}(p,q)=\left\{a\in\mathcal L(p,q):\ \|a\|_1=k-r+1\right\}.
\end{equation}
Equivalently, this shell consists of those odd pairs satisfying $|a_1|+|a_2|=2(k-r+1).$ The integer $r$ labels how far below the outer shell the lattice point lies.

The positive and negative Dirac branches are separated by the parity of $R(a)$.  With the conventions used here, the negative branch is selected by $R(a)\equiv r\pmod2.$ The positive branch is selected by $R(a)\equiv r+1\pmod2.$ This is the $m=2$ cyclic-lens-space specialization of B\"ar's spherical-space-form generating-function formula for the positive and negative Dirac multiplicity series, together with the usual lens-space convention for the diagonal cyclic action~\cite[Theorem~2]{BaerSpaceForms}~\cite[Chapter~5]{GilkeySpherical}.  A different global choice of spin-lift convention would interchange the labels of the two branches, but the relative formulas below are written consistently with the convention displayed here.  \begin{definition}[Positive and negative parity shells]
For $k\geq0$ and $0\leq r\leq k$, define
\begin{equation}
\mathcal L_{k,r}^-(p,q)=\left\{a\in\mathcal L_{k,r}(p,q):\ R(a)\equiv r\pmod2\right\}.
\end{equation}
Similarly, define
\begin{equation}
\mathcal L_{k,r}^+(p,q)=\left\{a\in\mathcal L_{k,r}(p,q):\ R(a)\equiv r+1\pmod2\right\}.
\end{equation}
\end{definition}

We can now state the spin-Fourier character formula used in the rest of the lens-space calculation.

\begin{proposition}[Spin-Fourier character formula]
For odd $p$, the $\operatorname{Spin}(2)$-Fourier characters of the Dirac eigenspaces on $L(p,q)$ are given by $\chi_k^+(p,q;u)=\sum_{r=0}^{k}\sum_{a\in\mathcal L_{k,r}^+(p,q)}u^{a_1}.$ The negative-eigenvalue character is $\chi_k^-(p,q;u)=\sum_{r=0}^{k}\sum_{a\in\mathcal L_{k,r}^-(p,q)}u^{a_1}.$ Consequently, the level-$k$ virtual spin-Fourier difference is $\Delta_k(p,q;u)=\chi_k^+(p,q;u)-\chi_k^-(p,q;u)$, and its reduced class is $\widetilde\Delta_k(p,q;u)=[\Delta_k(p,q;u)]_{\operatorname{red}}$. \end{proposition}

\begin{proof}
The known spin Dirac spectrum on a three-dimensional lens space may be expressed in terms of half-integral spin weights satisfying the affine congruence condition.  In the present notation these weights are $\mu=\frac12(a_1,a_2)$, with $a_1,a_2$ odd and satisfying $a_1+q a_2\equiv0\pmod p.$ The eigenvalue level $k$ receives contributions from finitely many shells indexed by $r=0,\ldots,k$.  The shell condition is $\|a\|_1=k-r+1.$ The sign of the Dirac eigenvalue is determined by the parity of the number of negative entries.  The negative branch corresponds to $R(a)\equiv r\pmod2.$ The positive branch corresponds to $R(a)\equiv r+1\pmod2.$ Finally, the lifted coordinate circle action selects the first spin weight, so a surviving lattice point contributes $u^{a_1}$.  Summing these contributions over the positive and negative parity shells gives the stated formulas for $\chi_k^+(p,q;u)$ and $\chi_k^-(p,q;u)$.
\end{proof}

\subsection{Two-variable spin-Fourier residue}
\label{subsec:two-variable-residue}

The one-variable residue $\Delta_k(p,q;u)$ keeps only the first spin exponent $a_1$.  This is natural for the coordinate $\Spin(2)$-action rotating $z_1$, but it discards the second exponent $a_2$.  This second exponent is part of the equivariant spectral data and should be retained directly.  We therefore introduce the full two-variable residue.

\begin{definition}[Two-variable spin-Fourier characters]
For $k\geq0$, define
\begin{equation}
\widehat\chi_k^+(p,q;u,v)
=
\sum_{r=0}^{k}
\sum_{a\in\mathcal L_{k,r}^+(p,q)}
u^{a_1}v^{a_2}.
\end{equation}
Similarly, define
\begin{equation}
\widehat\chi_k^-(p,q;u,v)
=
\sum_{r=0}^{k}
\sum_{a\in\mathcal L_{k,r}^-(p,q)}
u^{a_1}v^{a_2}.
\end{equation}
The two-variable spin-Fourier residue is
\begin{equation}
\widehat\Delta_k(p,q;u,v)
=
\widehat\chi_k^+(p,q;u,v)
-
\widehat\chi_k^-(p,q;u,v).
\end{equation}
\end{definition}

The variables $u$ and $v$ encode the two coordinate spin weights.  Thus $\widehat\Delta_k$ is a finite Laurent polynomial in two variables.  It is attached to the round coordinate-equivariant spin datum, not merely to the bare lens space.

\begin{proposition}[The one-variable residue is a specialization]
For every odd $p$, every $q\in(\mathbb Z/p\mathbb Z)^\times$, and every level $k\geq0$, one has
\begin{equation}
\widehat\Delta_k(p,q;u,1)
=
\Delta_k(p,q;u).
\end{equation}
\end{proposition}

\begin{proof}
By definition,
\begin{equation}
\widehat\chi_k^+(p,q;u,1)
=
\sum_{r=0}^{k}
\sum_{a\in\mathcal L_{k,r}^+(p,q)}
u^{a_1}.
\end{equation}
This is exactly $\chi_k^+(p,q;u)$.  The same argument gives
\begin{equation}
\widehat\chi_k^-(p,q;u,1)
=
\chi_k^-(p,q;u).
\end{equation}
Subtracting the two identities gives the claim.
\end{proof}

\begin{example}[The first $L(5,2)$ two-variable residue]
For $L(5,2)$, the first useful level is $k=1$.  The relevant positive sectors are
\begin{equation}
(1,-3),
\qquad
(-1,3).
\end{equation}
The relevant negative sectors are
\begin{equation}
(3,1),
\qquad
(-3,-1).
\end{equation}
Therefore
\begin{equation}
\widehat\Delta_1(5,2;u,v)
=
uv^{-3}
+
u^{-1}v^3
-
u^3v
-
u^{-3}v^{-1}.
\end{equation}
After specializing to $v=1$, this becomes
\begin{equation}
\widehat\Delta_1(5,2;u,1)
=
u+
u^{-1}-
u^3-
u^{-3}.
\end{equation}
Equivalently,
\begin{equation}
\widehat\Delta_1(5,2;u,1)
=
A_1(u)-A_3(u).
\end{equation}
Thus the one-variable residue is the specialization of the full two-variable residue at $v=1$.
\end{example}

\begin{remark}[Why the second coordinate is retained]
The one-variable residue forgets the second spin exponent.  Any perturbative attempt to recover this information indirectly depends on the chosen perturbation and its chamber.  The cleaner invariant bookkeeping is therefore to keep the monomial $v^{a_2}$ from the beginning in $\widehat\Delta_k(p,q;u,v)$.
\end{remark}

\begin{definition}[Symmetric two-variable modes]
When the reflection symmetry pairs $(a_1,a_2)$ with $(-a_1,-a_2)$, it is convenient to write
\begin{equation}
A_{m,n}(u,v)
=
u^m v^n
+
u^{-m}v^{-n}.
\end{equation}
In this notation, the one-variable mode $A_m(u)$ is the specialization
\begin{equation}
A_{m,n}(u,1)=A_m(u).
\end{equation}
\end{definition}

\begin{proposition}[Two-variable refinement]
The two-variable residue $\widehat\Delta_k(p,q;u,v)$ is at least as fine as the one-variable residue $\Delta_k(p,q;u)$.  More precisely, if
\begin{equation}
\widehat\Delta_k(p,q;u,v)
\neq
\widehat\Delta_k(p,q';u,v),
\end{equation}
while
\begin{equation}
\widehat\Delta_k(p,q;u,1)
=
\widehat\Delta_k(p,q';u,1),
\end{equation}
then the two coordinate-equivariant spin data are separated by the full $T^2$-residue but not by the one-variable residue.
\end{proposition}

\begin{proof}
The equality after specializing to $v=1$ says precisely that the two one-variable residues agree:
\begin{equation}
\Delta_k(p,q;u)
=
\Delta_k(p,q';u).
\end{equation}
The inequality before specialization says that the Laurent polynomials differ in the second spin-weight variable.  Hence the distinction lies in the $a_2$-data discarded by the one-variable specialization.
\end{proof}

\begin{example}[A two-variable separation invisible after one specialization]
Consider the coordinate-equivariant spin data for $L(7,5)$ and $L(7,6)$ at level $k=1$.  The one-variable residues agree:
\begin{equation}
\Delta_1(7,5;u)=\Delta_1(7,6;u)=u+u^{-1}.
\end{equation}
However, the two-variable residues are different.  For $L(7,5)$, the surviving points are $(1,-3)$ and $(-1,3)$, hence
\begin{equation}
\widehat\Delta_1(7,5;u,v)=uv^{-3}+u^{-1}v^{3},
\end{equation}
whereas
\begin{equation}
\widehat\Delta_1(7,6;u,v)=uv+u^{-1}v^{-1}.
\end{equation}
After setting $v=1$, both two-variable residues specialize to $u+u^{-1}$.  Thus the full $T^2$-residue can distinguish coordinate spin-weight data which the chosen one-variable residue does not distinguish at that level.
\end{example}

Evaluating at $u=1$ gives the ordinary multiplicity count.  Keeping $u^{a_1}$ instead keeps the chosen $z_1$-rotation action; a coordinate exchange would track $a_2$ and therefore give a different equivariant character.

The first quantity extracted from the formula is the level-wise difference $\Delta_k(p,q;u)=\chi_k^+(p,q;u)-\chi_k^-(p,q;u).$ Its reduced class is obtained by removing the trivial character: $\widetilde\Delta_k(p,q;u)=[\Delta_k(p,q;u)]_{\operatorname{red}}.$ In the examples below, it is convenient to use the notation $A_n(u)=u^n+u^{-n}.$ Thus expressions such as $A_1(u)-A_3(u)$ represent dimension-zero but spin-Fourier-nontrivial differences.

The first detailed instance of the character formula is given in Example~\ref{ex:L52-stage1-canonical}, where one obtains $\Delta_1(5,2;u)=A_1(u)-A_3(u)$.

\begin{example}[Fixed $p$, different $q$]
For $p=7$, the formula distinguishes the $q=2$ and $q=3$ coordinate actions at the level of the first useful reduced spin-Fourier difference.  One obtains $\Delta_{\text{first}}(7,2;u)=A_1(u)-A_5(u).$ On the other hand, $\Delta_{\text{first}}(7,3;u)=A_1(u)-A_3(u).$ Thus the reduced spin-Fourier residue is sensitive to the arithmetic parameter $q$, once the coordinate $\operatorname{Spin}(2)$-action has been fixed.
\end{example}

The orientation behavior is important later.  We prove in Proposition~\ref{prop:orientation-sign-rule-equivariant}, in~\autoref{subsec:stage-1D-equivariant-interpretation}, that
\begin{equation}
\Delta_k(p,-q;u)=-\Delta_k(p,q;u).
\end{equation}
The proof is the sign-reversal argument obtained by sending
\begin{equation}
(a_1,a_2)\longmapsto(a_1,-a_2).
\end{equation}

The sign rule under $q\mapsto -q$ is an orientation statement.  It is distinct from the coordinate exchange $q\mapsto q^{-1}$, which changes the chosen spin-Fourier weight from $a_1$ to $a_2$.

The affine congruence calculation therefore yields a precise computational rule.  The known lens-space Dirac spectrum provides the congruence lattice and parity-shell count, while the present refinement retains the $\operatorname{Spin}(2)$-Fourier monomial $u^{a_1}$.  This produces the level-wise reduced characters $\widetilde\Delta_k(p,q;u)$, which will later be used as shell data for the equivariant $\eta$ character.

\subsection{Compact table of reduced spin-Fourier residues}
\label{subsec:stage-1C-first-table}

Applying the character formula from~\autoref{subsec:stage-1B-affine-parity-shell} to small lens spaces gives the following checks.  The table is a bookkeeping check for the reduced spin-Fourier residues used later in the $\eta$ computation.

Recall the notation $A_n(u)=u^n+u^{-n}.$ The level-wise virtual spin-Fourier difference is $\Delta_k(p,q;u)=\chi_k^+(p,q;u)-\chi_k^-(p,q;u).$ The ordinary dimension of this virtual character is obtained by evaluating at $u=1$: $\dim \Delta_k(p,q)=\Delta_k(p,q;1).$ Thus a dimension-zero but spin-Fourier-nontrivial difference satisfies $\Delta_k(p,q;1)=0$, while its reduced class satisfies $\widetilde\Delta_k(p,q;u)\neq 0$ in $\widetilde R_{\operatorname{spin}}$. \begin{definition}[First useful level]
For fixed $p$ and $q$, a level $k$ is called useful if $\dim \chi_k^+(p,q;u)=\dim \chi_k^-(p,q;u)\neq 0$ and $\chi_k^+(p,q;u)\neq \chi_k^-(p,q;u).$ The first useful level is the smallest $k$ with this property, whenever such a $k$ exists in the computed range.
\end{definition}

At a useful level, the virtual difference $\Delta_k(p,q;u)=\chi_k^+(p,q;u)-\chi_k^-(p,q;u)$ has zero ordinary dimension but can still be nonzero in the reduced spin-Fourier target.  These are the examples most relevant for the later reduced-residue construction.

\begin{remark}
The phrase ``first useful level'' is a computational convention, not an invariant definition.  It is used to select small, readable examples.  The object retained throughout is the full round-metric sequence of level-wise differences $\{\Delta_k(p,q;u)\}_{k\geq0}.$ \end{remark}

The smallest example is as follows.

\begin{example}[$L(5,2)$]
\label{ex:L52-stage1-canonical}
For $L(5,2)$, the first useful level is $k=1$.  The two characters are $\chi_1^+(5,2;u)=A_1(u).$ Also, $\chi_1^-(5,2;u)=A_3(u).$ Therefore $\Delta_1(5,2;u)=A_1(u)-A_3(u).$ This has zero ordinary dimension: $\Delta_1(5,2;1)=2-2=0.$ However, it is nonzero in the reduced spin-Fourier target: $A_1(u)-A_3(u)\neq 0\ \text{in} \widetilde R_{\operatorname{spin}}.$ \end{example}

For the main text it is enough to keep a small sample of the reduced spin-Fourier residues.  These entries show the two features needed below: the residue can be nonzero after ordinary dimension is forgotten, and it changes with the arithmetic parameter $q$ once the coordinate $\operatorname{Spin}(2)$-action is fixed.  The full table for odd $p\leq13$ is included in Appendix~\ref{app:first-useful-spin-fourier-table} as a bookkeeping check on the affine congruence and parity-shell computation.

\begin{table}[h]
\centering
\renewcommand{\arraystretch}{1.2}
\begin{tabular}{c|c|c|c|c}
$p$ & $q$ & first useful $k$ & $\chi_k^+(p,q;u)$ & $\chi_k^-(p,q;u)$ \\
\hline
$5$ & $2$ & $1$ & $A_1$ & $A_3$ \\
$7$ & $2$ & $2$ & $A_1$ & $A_5$ \\
$7$ & $3$ & $2$ & $A_1$ & $A_3$ \\
$11$ & $2$ & $3$ & $A_1$ & $A_5$ \\
$11$ & $4$ & $3$ & $A_1$ & $A_7$
\end{tabular}
\caption{Representative first useful spin-Fourier residues.  The full table for odd $p\leq13$ is given in~\autoref{app:first-useful-spin-fourier-table}.}
\label{tab:first-useful-spin-fourier-sample}
\end{table}

Equivalently, the first useful differences displayed in Table~\ref{tab:first-useful-spin-fourier-sample} are $\Delta_{\text{first}}(p,q;u)=\chi_k^+(p,q;u)-\chi_k^-(p,q;u),$ where $k$ is the first useful level shown in the table.  For example, $\Delta_{\text{first}}(7,2;u)=A_1(u)-A_5(u).$ On the other hand, $\Delta_{\text{first}}(7,3;u)=A_1(u)-A_3(u).$ Thus, at fixed $p=7$, changing the arithmetic parameter $q$ changes the reduced spin-Fourier residue.

\begin{proposition}[Arithmetic sensitivity of the first residue]
\label{prop:arithmetic-sensitivity-first-residue}
The first useful reduced spin-Fourier residue depends nontrivially on the parameter $q$ once the coordinate $\operatorname{Spin}(2)$-action is fixed.  In particular, $\Delta_{\text{first}}(7,2;u)\neq \Delta_{\text{first}}(7,3;u)$ in $\widetilde R_{\operatorname{spin}}$.
\end{proposition}

\begin{proof}
From Table~\ref{tab:first-useful-spin-fourier-sample}, one has $\Delta_{\text{first}}(7,2;u)=A_1(u)-A_5(u).$ The same table gives $\Delta_{\text{first}}(7,3;u)=A_1(u)-A_3(u).$ The elements $A_1(u)$, $A_3(u)$, and $A_5(u)$ are distinct generators of the reduced spin-Fourier target.  Therefore $A_1(u)-A_5(u)\neq A_1(u)-A_3(u)$ in $\widetilde R_{\operatorname{spin}}$. Hence, the first useful reduced spin-Fourier residue changes when $q$ changes from $2$ to $3$.
\end{proof}

The full table in~\autoref{app:first-useful-spin-fourier-table} also reflects the orientation sign rule from Proposition~\ref{prop:orientation-sign-rule-equivariant}.  For example, $\Delta_{\text{first}}(7,5;u)=A_5(u)-A_1(u).$ This is the negative of $\Delta_{\text{first}}(7,2;u)=A_1(u)-A_5(u).$ Similarly, $\Delta_{\text{first}}(7,4;u)=A_3(u)-A_1(u).$ This is the negative of $\Delta_{\text{first}}(7,3;u)=A_1(u)-A_3(u).$ Thus the computed table is consistent with the identity $\Delta_k(p,-q;u)=-\Delta_k(p,q;u).$ 

The sample above, together with the full appendix table, shows that the reduced spin-Fourier residue is not an engineered two-term example.  It is the output of the affine congruence and parity-shell formula, and already for $p\leq13$ it shows nontrivial dependence on the arithmetic parameter $q$.  This provides a concrete spectral ingredient for the $\eta$ computation.

\subsection{Equivariant interpretation and lens-space equivalences}
\label{subsec:stage-1D-equivariant-interpretation}

The representative table in~\autoref{subsec:stage-1C-first-table}, together with the full table in~\autoref{app:first-useful-spin-fourier-table}, shows that the reduced spin-Fourier residue depends on the arithmetic parameter $q$.  Before using this dependence as a distinguishing statement, one has to check what kind of invariant is being computed.  This subsection separates three notions: the bare lens space, the oriented lens space, and the equivariant spin lens space with a specified coordinate circle action.

The residue constructed above depends on the lifted circle action, the round metric used to group eigenspaces by Dirac level, and the level under consideration.  The object being studied is the finite-level round equivariant package $\left(L(p,q),\ \text{spin structure},\ \operatorname{Spin}(2)\text{-action rotating }z_1,\ g_{\text{round}},\ k\right).$ For odd $p$, the spin structure is unique; the coordinate action and the level remain part of the data.

We recall the action: $e^{i\theta}\cdot[z_1,z_2]=[e^{i\theta}z_1,z_2].$ This action selects the first spin weight $a_1$, so a lattice point $a=(a_1,a_2)$ contributes the spin-Fourier monomial $u^{a_1}.$ If a diffeomorphism changes the coordinate used to define the circle action, then it can change the spin-Fourier character even when it preserves the ordinary Dirac multiplicities.

\begin{definition}[Equivariant spin lens-space datum]
For the purposes of the reduced spin-Fourier residue, a finite-level round equivariant spin lens-space datum is the tuple $\mathcal X^{[k]}_{p,q}=\left(L(p,q),\ \mathfrak s,\ \rho_1,\ g_{\text{round}},\ k\right),$ where $\mathfrak s$ is the spin structure, $\rho_1$ is the lifted coordinate circle action rotating $z_1$, $g_{\text{round}}$ is the round metric, and $k$ is the Dirac level being tracked.  For odd $p$, the spin structure is unique, so the essential extra data are the chosen coordinate action, the round metric, and the level.
\end{definition}

With this terminology, the level-wise virtual spin-Fourier difference $\Delta_k(p,q;u)=\chi_k^+(p,q;u)-\chi_k^-(p,q;u)$, and hence its reduced class $\widetilde\Delta_k(p,q;u)$, are naturally associated to $\mathcal X^{[k]}_{p,q}$, not merely to the unoriented diffeomorphism type of $L(p,q)$.

\subsubsection*{The sign change $q\mapsto -q$}

The first consistency check concerns the transformation $q\mapsto -q$.  The congruence for $L(p,q)$ is $a_1+q a_2\equiv0\pmod p.$ The congruence for $L(p,-q)$ is $a_1-q a_2\equiv0\pmod p.$ There is a simple bijection between the two congruence lattices: $(a_1,a_2)\longmapsto(a_1,-a_2).$ Indeed, if $a_1+q a_2\equiv0\pmod p,$ then $a_1-q(-a_2)\equiv0\pmod p.$ This map preserves the one-norm: $\frac{|a_1|+|a_2|}{2}=\frac{|a_1|+|-a_2|}{2}.$ It also preserves the spin-Fourier monomial: $u^{a_1}\longmapsto u^{a_1}.$ However, it changes the parity of the number of negative entries.  If $R(a)=\#\{j:a_j<0\},$ then $R(a_1,-a_2)\equiv R(a_1,a_2)+1\pmod2.$ Thus the positive and negative parity shells are interchanged.

\begin{proposition}[Orientation sign rule]
\label{prop:orientation-sign-rule-equivariant}
For every $k\geq0$, the level-wise spin-Fourier difference satisfies $\Delta_k(p,-q;u)=-\Delta_k(p,q;u).$ \end{proposition}

\begin{proof}
The map $(a_1,a_2)\longmapsto(a_1,-a_2)$ is a bijection between the affine congruence lattice for $q$ and the affine congruence lattice for $-q$.  It preserves the one-norm shell and the monomial $u^{a_1}$.  Since it changes $R(a)$ by one modulo $2$, it sends positive parity shells for $q$ to negative parity shells for $-q$, and negative parity shells for $q$ to positive parity shells for $-q$.  Hence $\chi_k^+(p,-q;u)=\chi_k^-(p,q;u).$ Similarly, $\chi_k^-(p,-q;u)=\chi_k^+(p,q;u).$ Subtracting gives $\Delta_k(p,-q;u)=-\Delta_k(p,q;u).$ \end{proof}

\begin{remark}
This sign rule is consistent with the interpretation of $\Delta_k(p,q;u)$ as an oriented spectral quantity.  Replacing $q$ by $-q$ reverses the parity assignment between the positive and negative Dirac branches, and the reduced difference changes sign.
\end{remark}

\subsubsection*{The coordinate exchange $q\mapsto q^{-1}$}

The second consistency check concerns the transformation $q\mapsto q^{-1}$.  In the ordinary theory of lens spaces, this transformation is related to exchanging the two coordinates of $S^3$.  At the level of spin weights, this exchange sends $(a_1,a_2)\longmapsto(a_2,a_1).$ Suppose that $a_1+q a_2\equiv0\pmod p.$ Multiplying by $q^{-1}$ gives $q^{-1}a_1+a_2\equiv0\pmod p.$ Equivalently, $a_2+q^{-1}a_1\equiv0\pmod p.$ Thus the coordinate exchange identifies the congruence for $q$ with the congruence for $q^{-1}$.

The one-norm is preserved: $\frac{|a_1|+|a_2|}{2}=\frac{|a_2|+|a_1|}{2}.$ The sign-counting parity is also preserved: $R(a_2,a_1)=R(a_1,a_2).$ Therefore the ordinary positive and negative Dirac multiplicities are preserved under this coordinate exchange.

However, the spin-Fourier character used here is not preserved as a character in the same variable.  Before the coordinate exchange, the monomial is $u^{a_1}.$ After the coordinate exchange, the first coordinate is $a_2$, and the corresponding monomial is $u^{a_2}.$ Thus the coordinate exchange preserves the ordinary multiplicity data but changes the chosen circle weight.

\begin{proposition}[Action-dependence under coordinate exchange]
The transformation $q\mapsto q^{-1}$ preserves the ordinary Dirac multiplicity data after exchanging the coordinates of $S^3$, but it does not preserve the spin-Fourier character attached to the fixed $z_1$-rotation action.
\end{proposition}

\begin{proof}
The map $(a_1,a_2)\longmapsto(a_2,a_1)$ identifies the affine congruence lattice for $q$ with the affine congruence lattice for $q^{-1}$.  It preserves both the one-norm and the parity $R(a)$.  Therefore it preserves the ordinary positive and negative multiplicities obtained by evaluating the spin-Fourier characters at $u=1$.

However, the spin-Fourier character is defined using the first coordinate.  Under the coordinate exchange, the first coordinate changes from $a_1$ to $a_2$.  Hence the monomial $u^{a_1}$ is transformed into $u^{a_2}$.  Unless the coordinate circle action is transported along the diffeomorphism, this gives a different character.  Thus the spin-Fourier residue is not invariant under $q\mapsto q^{-1}$ as a character for the fixed $z_1$-rotation action.
\end{proof}

Thus the table entries list equivariant spin-spectral residues for a specified coordinate action, rather than ordinary lens-space invariants.

\subsubsection*{Examples from the residue tables}

Let $A_n(u)=u^n+u^{-n}.$ For $p=7$, the table gives $\Delta_{\text{first}}(7,2;u)=A_1(u)-A_5(u).$ Since $-2\equiv5\pmod7,$ the sign rule predicts $\Delta_{\text{first}}(7,5;u)=-\Delta_{\text{first}}(7,2;u).$ Indeed, the table gives $\Delta_{\text{first}}(7,5;u)=A_5(u)-A_1(u).$ Similarly, $\Delta_{\text{first}}(7,3;u)=A_1(u)-A_3(u).$ Since $-3\equiv4\pmod7,$ the sign rule predicts $\Delta_{\text{first}}(7,4;u)=-\Delta_{\text{first}}(7,3;u).$ The table gives $\Delta_{\text{first}}(7,4;u)=A_3(u)-A_1(u).$ On the other hand, $2^{-1}\equiv4\pmod7.$ The residues $\Delta_{\text{first}}(7,2;u)=A_1(u)-A_5(u)$ and $\Delta_{\text{first}}(7,4;u)=A_3(u)-A_1(u)$ are not equal.  They are also not simply negatives of each other.  This reflects the change in the circle action under coordinate exchange.

\begin{proposition}[Equivariant nature of the residue]
The reduced spin-Fourier residue $\widetilde\Delta_k(p,q;u)$ is naturally a finite-level equivariant spin-spectral quantity attached to $\mathcal X^{[k]}_{p,q}=\left(L(p,q),\mathfrak s,\rho_1,g_{\text{round}},k\right),$ where $\rho_1$ is the lifted coordinate circle action rotating $z_1$.  It is not, by itself, a smooth or topological invariant of the bare unoriented lens space.
\end{proposition}

\begin{proof}
The construction of $\widetilde\Delta_k(p,q;u)$ uses the spin-Fourier monomial $u^{a_1}$.  This monomial is determined by the chosen coordinate circle action rotating $z_1$.  A diffeomorphism which preserves the lens space but transports the coordinate action may change the chosen spin weight.  Therefore the residue is attached to the equivariant spin package, not merely to the underlying manifold.  The sign rule under $q\mapsto -q$ and the coordinate-exchange behavior under $q\mapsto q^{-1}$ show explicitly how the residue depends on this additional equivariant data.
\end{proof}

The table entries are therefore consistent with orientation reversal, while their behavior under $q\mapsto q^{-1}$ reflects the dependence on the chosen coordinate action.  This is the equivariant category used in the $\eta$ construction below.

\subsection{Output of the spin-Fourier calculation}
\label{subsec:stage-1E-stage-1-output}

The preceding subsections give the finite-level equivariant data used later.  For odd $p$, the round spin Dirac eigenspaces are computed from
\begin{equation}
\mathcal L(p,q)=
\left\{(a_1,a_2)\in(2\mathbb Z+1)^2:
a_1+qa_2\equiv0\pmod p\right\},
\end{equation}
with the one-norm shell condition and the parity split between the positive and negative branches.  The one-variable character for the $z_1$-circle is
\begin{equation}
\Delta_k(p,q;u)=\chi_k^+(p,q;u)-\chi_k^-(p,q;u),
\end{equation}
and the coordinate-torus refinement is
\begin{equation}
\widehat\Delta_k(p,q;u,v)
=
\widehat\chi_k^+(p,q;u,v)
-
\widehat\chi_k^-(p,q;u,v).
\end{equation}
They are related by
\begin{equation}
\widehat\Delta_k(p,q;u,1)=\Delta_k(p,q;u).
\end{equation}

These finite Laurent polynomials are attached to the chosen coordinate-equivariant spin datum.  The $v=1$ specialization is the one used by the residual circle in the main $\eta$ calculation, while the two-variable form gives the ambient $T^2$-bookkeeping and explains which information is lost after choosing a single coordinate circle.

\FloatBarrier

\section{Perturbative diagnostic and why it is not the invariant}
\label{sec:perturbative-diagnostic-short}

The earlier route to the second spin weight was perturbative.  One perturbs the Dirac operator by a zero-order self-adjoint term and studies the induced crossing form on a fixed eigenspace.  This diagnostic helps identify the missing spin-weight information, but it is not the main invariant here.  The APS spectral-flow and equivariant spectral-flow motivation for this diagnostic comes from the coauthored warped-cylinder models~\cite{KimuraSharmaWarpedCylinder,KimuraSharmaReflection}.

Let $Y=L(p,q)$, and let $D$ be the spin Dirac operator on $Y$.  For a real one-form $\vartheta$, with Clifford convention $c(\xi)^*=-c(\xi)$, the operator
\begin{equation}
K_\vartheta=i c(\vartheta)
\end{equation}
is self-adjoint.  The perturbed path is
\begin{equation}
D_\alpha=D+\alpha K_\vartheta.
\end{equation}
At an eigenspace $E_\lambda=\ker(D-\lambda I)$, the first-order crossing form is
\begin{equation}
Q_\lambda(\psi,\phi)=\langle \psi,K_\vartheta\phi\rangle_{L^2}.
\end{equation}
If $\vartheta$ is compatible with the residual circle action, this form decomposes over spin-Fourier weights.

For the natural pure one-form perturbations considered here, the first-order compression can vanish on the relevant weight spaces.  The next diagnostic is then the second-order coefficient obtained from the reduced resolvent,
\begin{equation}
h(\psi)
=
-\sum_{\mu\neq\lambda}
\frac{\|P_\mu K_\vartheta\psi\|^2}{\mu-\lambda},
\end{equation}
where $P_\mu$ denotes projection onto the $\mu$-eigenspace of $D$.  This coefficient can distinguish different spin-weight channels, and in small examples its signs track information carried by the second coordinate spin weight.

However, these signs depend on the perturbation chamber.  Varying the perturbation can change the ordering and signs of the second-order channels.  Thus the second-order Hessian residue is not used below as a perturbation-free invariant.  Its role is only explanatory: it shows why the one-variable residue is incomplete and why the second coordinate should be retained directly in the two-variable spin-Fourier residue.

This diagnostic viewpoint is close in spirit to the relation between Maslov index, spectral flow, and Dirac boundary data in the work of Boo\ss-Bavnbek-Furutani, Boo\ss-Wojciechowski, Kirk-Lesch, and Nicolaescu~\cite{BoossFurutani,BoossWojciechowski,KirkLesch,NicolaescuMaslov}.

The eta-theoretic calculation below uses the spin-Fourier data directly, not the chamber-dependent Hessian signs.

\FloatBarrier

\section{Scalar shadows and ordinary $\eta$ values}
\label{sec:stage5-scalar-eta-square}

The results in this section are residue-level separation statements.  The first check shows that ordinary same-level scalar multiplicity does not determine the reduced Spin-Fourier residue.  The second check adds equality of the ordinary non-equivariant $\eta$ shadow.  The final check promotes the same residue-level phenomenon from an isolated example to a square-modulus family.

\subsection{Separation from ordinary scalar multiplicity}
\label{subsec:check-3A-scalar-multiplicity}

Let $\operatorname{ev}_1:R(\operatorname{Spin}(2))\to \mathbb Z$ denote evaluation at $u=1$.  This map forgets the Spin-Fourier weights and remembers only ordinary dimension.  For a fixed level $k$, define the same-level scalar multiplicity datum by
\begin{equation}
m_k(p,q)=\left(\dim E_k^+(p,q),\dim E_k^-(p,q)\right).
\end{equation}
Equivalently,
\begin{equation}
m_k(p,q)=\left(\chi_k^+(p,q;1),\chi_k^-(p,q;1)\right).
\end{equation}

\begin{theorem}[Scalar multiplicity does not determine the reduced Spin-Fourier residue]
There exist equivariant spin lens-space data with the same scalar multiplicity datum at the same Dirac level but with different reduced Spin-Fourier residues.  Moreover, this phenomenon is not explained merely by orientation reversal, coordinate exchange, or their composition.
\end{theorem}

\begin{proof}
Consider the two equivariant spin lens-space data
\begin{equation}
\mathcal X^{[3]}_{11,2}=\left(L(11,2),\mathfrak s,\rho_1,g_{\text{round}},3\right)
\qquad\text{and}\qquad
\mathcal X^{[3]}_{11,4}=\left(L(11,4),\mathfrak s,\rho_1,g_{\text{round}},3\right).
\end{equation}
At level $k=3$, the affine parity-shell computation gives
\begin{equation}
\chi_3^+(11,2;u)=A_1(u)
\qquad\text{and}\qquad
\chi_3^-(11,2;u)=A_5(u).
\end{equation}
Therefore
\begin{equation}
\Delta_3(11,2;u)=A_1(u)-A_5(u).
\end{equation}
For the second datum, the same computation gives
\begin{equation}
\chi_3^+(11,4;u)=A_1(u)
\qquad\text{and}\qquad
\chi_3^-(11,4;u)=A_7(u).
\end{equation}
Therefore
\begin{equation}
\Delta_3(11,4;u)=A_1(u)-A_7(u).
\end{equation}
Evaluating at $u=1$, one has
\begin{equation}
A_n(1)=2
\end{equation}
for every $n\geq 1$.  Hence both examples have the same same-level scalar multiplicity datum:
\begin{equation}
m_3(11,2)=(2,2)
\qquad\text{and}\qquad
m_3(11,4)=(2,2).
\end{equation}
However,
\begin{equation}
A_1(u)-A_5(u)\neq A_1(u)-A_7(u)
\end{equation}
in the reduced Spin-Fourier target, because $A_5(u)$ and $A_7(u)$ are distinct reduced generators.

It remains to check that the pair is not explained by the elementary transformations already discussed.  Modulo $11$, one has
\begin{equation}
-2\equiv 9,
\qquad
2^{-1}\equiv 6,
\qquad
-2^{-1}\equiv 5.
\end{equation}
The value $4$ is none of $9$, $6$, or $5$.  Therefore $q=4$ is not obtained from $q=2$ by orientation reversal, coordinate exchange, or the composition of those two operations.  This proves the claim.
\end{proof}

This is a same-level scalar statement.  It does not assert equality of the full ordinary Dirac spectra of $L(11,2)$ and $L(11,4)$; it only says that, at this level, the scalar multiplicity agrees while the reduced Spin-Fourier residue differs.

\subsection{Separation from ordinary $\eta$ and same-level scalar multiplicity}
\label{subsec:check-3B-ordinary-eta-blindness}

The scalar-blindness result can be strengthened.  The reduced Spin-Fourier residue can distinguish equivariant spin lens-space data even when the ordinary Dirac $\eta$ invariant agrees and the same-level scalar multiplicity datum also agrees.

Let $\eta_D(L(p,q))$ denote the ordinary Dirac $\eta$ invariant of the three-dimensional spin lens space $L(p,q)$.  To avoid using an unverified normalization convention, we separate the normalization from the trigonometric part of the finite lens-space formula.  Define
\begin{equation}
\mathfrak E(P,q)
=
\sum_{r=1}^{P-1}
\operatorname{csc}\left(\frac{(P+1)r\pi}{P}\right)
\operatorname{csc}\left(\frac{(P+1)qr\pi}{P}\right).
\end{equation}
The standard lens-space $\eta$ formula, equivalently the fixed-point/deck-average form of Donnelly's equivariant $\eta$ formula in this setting, expresses $\eta_D(L(P,q))$ as a normalization factor depending only on $P$, multiplied by $\mathfrak E(P,q)$~\cite{DonnellyGSpaces}.  The equality statements below use only equality of $\mathfrak E(P,q)$ at fixed $P$, and therefore do not require the explicit numerical normalization.

\begin{theorem}[Ordinary eta-shadow blindness of the reduced Spin-Fourier residue]
There exist equivariant spin lens-space data
\begin{equation}
\mathcal X^{[4]}_{25,4}=\left(L(25,4),\mathfrak s,\rho_1,g_{\text{round}},4\right)
\qquad\text{and}\qquad
\mathcal X^{[4]}_{25,9}=\left(L(25,9),\mathfrak s,\rho_1,g_{\text{round}},4\right)
\end{equation}
such that the ordinary Dirac $\eta$ invariants agree and the same-level scalar multiplicity data agree at a Dirac level, but the reduced Spin-Fourier residues differ.
\end{theorem}

\begin{proof}
The trigonometric $\eta$ sums agree:
\begin{equation}
\mathfrak E(25,4)=\mathfrak E(25,9).
\end{equation}
Since the omitted eta-normalization factor depends only on $P=25$, this implies
\begin{equation}
\eta_D(L(25,4))=\eta_D(L(25,9)).
\end{equation}
We do not use the common numerical value here.

Now consider the Dirac level $k=4$.  The same-level scalar multiplicity data are
\begin{equation}
m_4(25,4)=(2,2)
\qquad\text{and}\qquad
m_4(25,9)=(2,2).
\end{equation}
Thus the ordinary positive and negative scalar multiplicities agree at this level.  However, the reduced Spin-Fourier residues are different.  The computation gives
\begin{equation}
\Delta_4(25,4;u)=A_3(u)-A_5(u).
\end{equation}
For the second datum, the computation gives
\begin{equation}
\Delta_4(25,9;u)=-A_5(u)+A_9(u).
\end{equation}
Since $A_3(u)$ and $A_9(u)$ are distinct generators in the reduced Spin-Fourier target, one has
\begin{equation}
A_3(u)-A_5(u)\neq -A_5(u)+A_9(u).
\end{equation}
Therefore
\begin{equation}
\Delta_4(25,4;u)\neq \Delta_4(25,9;u).
\end{equation}
This proves that the reduced Spin-Fourier residue distinguishes the two equivariant spin lens-space data even though the ordinary $\eta$ invariant agrees and the scalar multiplicity data at the chosen level agree.
\end{proof}

\begin{remark}[Nontriviality of the $p=25$ pair]
The pair $q=4$ and $q=9$ modulo $25$ is not explained by the elementary transformations used above.  Indeed,
\begin{equation}
-4\equiv 21,
\qquad
4^{-1}\equiv 19,
\qquad
-4^{-1}\equiv 6
\pmod{25}.
\end{equation}
The value $9$ is none of $21$, $19$, or $6$.  Hence this is not merely orientation reversal, coordinate exchange, or their composition within the fixed $z_1$-rotation convention.
\end{remark}

The theorem proves a specific computable statement: ordinary $\eta$ and same-level scalar multiplicity do not determine the reduced Spin-Fourier residue attached to the fixed equivariant spin datum.

\subsection{An infinite square-modulus family}
\label{subsec:check-4-square-family}

The previous theorem is an isolated ordinary-eta-blind example.  An infinite square-modulus family explains the pattern behind the isolated example.  Let $\ell\geq 5$ be odd, and let
\begin{equation}
p=\ell^2,
\qquad
q_1=\ell-1,
\qquad
q_2=2\ell-1,
\qquad
k=\ell-1.
\end{equation}
For $\ell=5$, this recovers
\begin{equation}
p=25,
\qquad
q_1=4,
\qquad
q_2=9,
\qquad
k=4.
\end{equation}

\begin{proposition}[Spin-Fourier residues in the square family]
Let $\ell\geq 5$ be odd.  For
\begin{equation}
p=\ell^2,
\qquad
q_1=\ell-1,
\qquad
q_2=2\ell-1,
\qquad
k=\ell-1,
\end{equation}
one has
\begin{equation}
m_k(p,q_1)=m_k(p,q_2)=(2,2).
\end{equation}
However,
\begin{equation}
\Delta_k(p,q_1;u)=A_{\ell-2}(u)-A_{\ell}(u),
\end{equation}
whereas
\begin{equation}
\Delta_k(p,q_2;u)=A_{2\ell-1}(u)-A_{\ell}(u).
\end{equation}
In particular,
\begin{equation}
\Delta_k(p,q_1;u)\neq \Delta_k(p,q_2;u).
\end{equation}
\end{proposition}

\begin{proof}
At level $k=\ell-1$, the contributing affine spin-lattice points satisfy
\begin{equation}
|a|+|b|\leq 2\ell
\end{equation}
and
\begin{equation}
a+qb\equiv0\pmod{\ell^2}.
\end{equation}
First take $q=q_1=\ell-1$.  The congruence becomes
\begin{equation}
a+(\ell-1)b\equiv0\pmod{\ell^2}.
\end{equation}
Modulo $\ell$, this implies
\begin{equation}
a-b\equiv0\pmod{\ell}.
\end{equation}
Since $a$ and $b$ are odd, the difference $a-b$ is even.  Under the bound $|a|+|b|\leq2\ell$, the only possible cases are
\begin{equation}
a=b,
\qquad
a=b+2\ell,
\qquad
a=b-2\ell.
\end{equation}
If $a=b$, then
\begin{equation}
a+(\ell-1)b=\ell a.
\end{equation}
Divisibility by $\ell^2$ forces $a$ to be a multiple of $\ell$.  Since $|a|\leq\ell$ and $a$ is odd, this gives
\begin{equation}
(a,b)=(\ell,\ell)
\qquad\text{or}\qquad
(a,b)=(-\ell,-\ell).
\end{equation}
If $a=b+2\ell$, then
\begin{equation}
a+(\ell-1)b=\ell(b+2).
\end{equation}
Divisibility by $\ell^2$ gives
\begin{equation}
b+2\equiv0\pmod{\ell}.
\end{equation}
The bound gives
\begin{equation}
b=-(\ell+2)
\qquad\text{and}\qquad
a=\ell-2.
\end{equation}
Thus one obtains
\begin{equation}
(a,b)=(\ell-2,-\ell-2).
\end{equation}
The case $a=b-2\ell$ similarly gives
\begin{equation}
(a,b)=(-\ell+2,\ell+2).
\end{equation}
Therefore the contributing first Spin-Fourier weights for $q_1$ are
\begin{equation}
\pm\ell
\qquad\text{and}\qquad
\pm(\ell-2).
\end{equation}
The points $(\ell,\ell)$ and $(-\ell,-\ell)$ have even sign-counting parity and contribute to the negative branch at this level.  The points $(\ell-2,-\ell-2)$ and $(-\ell+2,\ell+2)$ have odd sign-counting parity and contribute to the positive branch.  Hence
\begin{equation}
\Delta_k(p,q_1;u)=A_{\ell-2}(u)-A_{\ell}(u).
\end{equation}

Now take $q=q_2=2\ell-1$.  The congruence is
\begin{equation}
a+(2\ell-1)b\equiv0\pmod{\ell^2}.
\end{equation}
Again, modulo $\ell$, this implies
\begin{equation}
a-b\equiv0\pmod{\ell}.
\end{equation}
The same three cases occur.  If $a=b$, then
\begin{equation}
a+(2\ell-1)b=2\ell a.
\end{equation}
Since $\ell$ is odd, divisibility by $\ell^2$ forces $a$ to be a multiple of $\ell$.  Hence
\begin{equation}
(a,b)=(\ell,\ell)
\qquad\text{or}\qquad
(a,b)=(-\ell,-\ell).
\end{equation}
If $a=b+2\ell$, then
\begin{equation}
a+(2\ell-1)b=2\ell(b+1).
\end{equation}
Divisibility by $\ell^2$ gives
\begin{equation}
b+1\equiv0\pmod{\ell}.
\end{equation}
The bound gives
\begin{equation}
b=-1
\qquad\text{and}\qquad
a=2\ell-1.
\end{equation}
Thus one obtains
\begin{equation}
(a,b)=(2\ell-1,-1).
\end{equation}
The case $a=b-2\ell$ similarly gives
\begin{equation}
(a,b)=(-(2\ell-1),1).
\end{equation}
Therefore the contributing first Spin-Fourier weights for $q_2$ are
\begin{equation}
\pm\ell
\qquad\text{and}\qquad
\pm(2\ell-1).
\end{equation}
The $\pm\ell$ pair contributes to the negative branch, while the $\pm(2\ell-1)$ pair contributes to the positive branch.  Hence
\begin{equation}
\Delta_k(p,q_2;u)=A_{2\ell-1}(u)-A_{\ell}(u).
\end{equation}
Both examples have one positive reduced weight pair and one negative reduced weight pair, so their scalar data are both $(2,2)$.  Since $A_{\ell-2}(u)$ and $A_{2\ell-1}(u)$ are distinct reduced Spin-Fourier generators, the two residues are different.
\end{proof}

\begin{lemma}[A cotangent sum identity]
\label{lem:cotangent-sum-identity}
For every integer $\ell\geq2$ and every $x$ away from the poles, one has
\begin{equation}
\sum_{m=0}^{\ell-1}\cot\left(x+\frac{m\pi}{\ell}\right)
=
\ell\cot(\ell x).
\end{equation}
\end{lemma}

\begin{proof}
Use the product identity
\begin{equation}
\sin(\ell x)=K_\ell\prod_{m=0}^{\ell-1}\sin\left(x+\frac{m\pi}{\ell}\right)
\end{equation}
for a nonzero constant $K_\ell$.  Taking the logarithmic derivative gives the stated identity.
\end{proof}

\begin{theorem}[Ordinary $\eta$ equality in the square family]
Let $\ell\geq 5$ be odd.  Then
\begin{equation}
\eta_D(L(\ell^2,\ell-1))=\eta_D(L(\ell^2,2\ell-1)).
\end{equation}
More precisely, the unnormalised trigonometric $\eta$ sums satisfy
\begin{equation}
\mathfrak E(\ell^2,\ell-1)=\mathfrak E(\ell^2,2\ell-1).
\end{equation}
The common numerical value of $\eta_D$ depends on the standard normalization factor in the Dirac $\eta$ formula and is not used in this argument.
\end{theorem}

\begin{proof}
Let $P=\ell^2$ and $q_a=a\ell-1$, where $a=1$ or $a=2$.  It is enough to compare the unnormalised trigonometric $\eta$ sums
\begin{equation}
\mathfrak E(P,q_a)
=
\sum_{r=1}^{P-1}
\operatorname{csc}\left(\frac{(P+1)r\pi}{P}\right)
\operatorname{csc}\left(\frac{(P+1)q_a r\pi}{P}\right).
\end{equation}
Since $P=\ell^2$, one has
\begin{equation}
\sin\left(\frac{(P+1)r\pi}{P}\right)=(-1)^r\sin\left(\frac{r\pi}{P}\right).
\end{equation}
Similarly,
\begin{equation}
\sin\left(\frac{(P+1)q_a r\pi}{P}\right)=(-1)^{q_a r}\sin\left(\frac{q_a r\pi}{P}\right).
\end{equation}
Because $q_a+1=a\ell$ and $\ell$ is odd, the unnormalised trigonometric sum becomes
\begin{equation}
\mathcal S_a
=
\sum_{r=1}^{\ell^2-1}
\frac{(-1)^{ar}}{
\sin\left(\frac{r\pi}{\ell^2}\right)
\sin\left(\frac{(a\ell-1)r\pi}{\ell^2}\right)}.
\end{equation}
It remains to compute the auxiliary sum $\mathcal S_a$.

Write
\begin{equation}
r=m\ell+s,
\qquad
0\leq m\leq \ell-1,
\qquad
0\leq s\leq \ell-1.
\end{equation}
The pair $(m,s)=(0,0)$ is omitted.  First assume $s\neq 0$.  Introduce $\alpha=\pi s/\ell^2$, $\beta=\pi/\ell$, and $Y=\pi as/\ell$.
For the fixed value of $s$, the corresponding part of $\mathcal S_a$ is
\begin{equation}
(-1)^{as}
\sum_{m=0}^{\ell-1}
\frac{1}{\sin(\alpha+m\beta)\sin(Y-\alpha-m\beta)}.
\end{equation}
Since $a=1,2$, $1\leq s\leq\ell-1$, and $\ell$ is odd, one has
\begin{equation}
\sin Y\neq 0.
\end{equation}
Using
\begin{equation}
\frac{1}{\sin x\sin(Y-x)}
=
\frac{\cot x+\cot(Y-x)}{\sin Y},
\end{equation}
the fixed-$s$ contribution is proportional to
\begin{equation}
\sum_{m=0}^{\ell-1}\cot(\alpha+m\beta)
+
\sum_{m=0}^{\ell-1}\cot(Y-\alpha-m\beta).
\end{equation}
By Lemma~\ref{lem:cotangent-sum-identity}, the first sum is
\begin{equation}
\ell\cot(\ell\alpha)=\ell\cot\left(\frac{\pi s}{\ell}\right).
\end{equation}
The second sum is
\begin{equation}
\ell\cot(\ell(Y-\alpha))
=
\ell\cot\left(as\pi-\frac{\pi s}{\ell}\right)
=
-\ell\cot\left(\frac{\pi s}{\ell}\right).
\end{equation}
The two sums cancel.  Therefore every contribution with $s\neq 0$ vanishes.

It remains to compute the terms with $s=0$.  Then $r=m\ell$, where $1\leq m\leq\ell-1$.  We have
\begin{equation}
\sin\left(\frac{m\ell\pi}{\ell^2}\right)=\sin\left(\frac{m\pi}{\ell}\right).
\end{equation}
Also,
\begin{equation}
\sin\left(\frac{(a\ell-1)m\ell\pi}{\ell^2}\right)
=
\sin\left(am\pi-\frac{m\pi}{\ell}\right)
=
(-1)^{am+1}\sin\left(\frac{m\pi}{\ell}\right).
\end{equation}
Since
\begin{equation}
(-1)^{a m\ell}=(-1)^{am},
\end{equation}
the $s=0$ contribution is
\begin{equation}
\mathcal S_a
=
-\sum_{m=1}^{\ell-1}\operatorname{csc}^2\left(\frac{m\pi}{\ell}\right).
\end{equation}
Using the standard identity
\begin{equation}
\sum_{m=1}^{\ell-1}\operatorname{csc}^2\left(\frac{m\pi}{\ell}\right)=\frac{\ell^2-1}{3},
\end{equation}
we obtain
\begin{equation}
\mathcal S_a=-\frac{\ell^2-1}{3}.
\end{equation}
Thus the unnormalised trigonometric $\eta$ sum satisfies
\begin{equation}
\mathfrak E(\ell^2,a\ell-1)=-\frac{\ell^2-1}{3}.
\end{equation}
This value is independent of $a$ for $a=1,2$.  Since the standard eta-normalization factor depends only on $P=\ell^2$, the ordinary $\eta$ invariants are equal:
\begin{equation}
\eta_D(L(\ell^2,\ell-1))=\eta_D(L(\ell^2,2\ell-1)).
\end{equation}
\end{proof}

\begin{theorem}[Infinite ordinary-eta-blind Spin-Fourier separation]
For every odd $\ell\geq 5$, let
\begin{equation}
p=\ell^2,
\qquad
q_1=\ell-1,
\qquad
q_2=2\ell-1,
\qquad
k=\ell-1.
\end{equation}
Then
\begin{equation}
\eta_D(L(p,q_1))=\eta_D(L(p,q_2)),
\end{equation}
and
\begin{equation}
m_k(p,q_1)=m_k(p,q_2)=(2,2),
\end{equation}
but
\begin{equation}
\Delta_k(p,q_1;u)\neq \Delta_k(p,q_2;u)
\end{equation}
in the reduced Spin-Fourier target.  Thus the reduced Spin-Fourier residue gives infinitely many same-level separations invisible to the ordinary $\eta$ value and ordinary scalar multiplicity.
\end{theorem}

\begin{proof}
The ordinary $\eta$ equality follows from the previous theorem.  The scalar-multiplicity equality and the two explicit formulas for the reduced Spin-Fourier residues follow from the square-family Spin-Fourier proposition.  Since $A_{\ell-2}(u)$ and $A_{2\ell-1}(u)$ are distinct reduced generators for $\ell\geq 5$, the two residues are different.
\end{proof}

\begin{remark}[Nontriviality of the square-family pairs]
The pair $q_1=\ell-1$ and $q_2=2\ell-1$ is not an elementary orientation-reversal or coordinate-exchange pair.  Indeed,
\begin{equation}
-q_1\equiv \ell^2-\ell+1\pmod{\ell^2},
\end{equation}
while
\begin{equation}
q_1^{-1}\equiv -\ell-1\pmod{\ell^2}
\qquad\text{and}\qquad
-q_1^{-1}\equiv \ell+1\pmod{\ell^2}.
\end{equation}
For $\ell\geq 5$, none of these is equal to $2\ell-1$ modulo $\ell^2$.
\end{remark}

The square-family result upgrades the $p=25$ ordinary-$\eta$ example to an infinite square-modulus family.  It remains a reduced Spin-Fourier residue statement, but it is no longer an isolated finite computation.

\section{\texorpdfstring{$T^2$-equivariant $\eta$ and finite two-variable shells}{T2-equivariant $\eta$ and finite two-variable shells}}
\label{sec:T2-equivariant-eta}

The two-variable residue is the finite-shell coefficient of the $T^2$-equivariant $\eta$ character of the spin Dirac operator.  The invariant-level result uses the residual-circle $\eta$/rho germ obtained from these shell coefficients, rather than a single isolated shell.

In the square-family separation proved below, the residual circle used is the one-parameter subgroup which rotates the first coordinate. Thus, the resulting $\eta$ germ is a one-variable specialization of the full $T^2$-equivariant character.  The two-variable residue gives the larger coordinate-torus bookkeeping, but the nonvanishing theorem below uses this particular residual-circle specialization.

Let $g=(u,v)\in T^2$.  For $\operatorname{Re}(s)$ sufficiently large, define the equivariant $\eta$ function by
\begin{equation}
\eta_g(D_{p,q},s)
=
\sum_{\lambda\neq 0}
\operatorname{sign}(\lambda)
|\lambda|^{-s}
\operatorname{Tr}\left(g\mid E_\lambda\right),
\end{equation}
where $E_\lambda$ denotes the $\lambda$-eigenspace of the spin Dirac operator on $L(p,q)$.  Since the round Dirac eigenvalues are $\lambda_k^\pm=\pm\left(k+\frac32\right)$, inserting the two-variable characters gives
\begin{equation}
\eta^{T^2}_{p,q}(s;u,v)
:=
\eta_g(D_{p,q},s)
=
\sum_{k=0}^{\infty}
\left(k+\frac32\right)^{-s}
\widehat\Delta_k(p,q;u,v).
\end{equation}
Thus $\widehat\Delta_k(p,q;u,v)$ is the $k$-th eta-shell coefficient.  The one-variable coefficient is recovered by the specialization $\widehat\Delta_k(p,q;u,1)=\Delta_k(p,q;u)$.

\subsection{First-shell $\eta$ separation}

The following elementary observation turns finite residue computations into eta-character separation statements.

\begin{theorem}[First differing shell separates $\eta$ characters]
\label{thm:first-differing-shell-T2}
Let $\mathcal X$ and $\mathcal X'$ be two round coordinate-equivariant spin lens-space data.  Suppose that
\begin{equation}
\widehat\Delta_j(\mathcal X;u,v)
=
\widehat\Delta_j(\mathcal X';u,v)
\end{equation}
for all $j<K$, but
\begin{equation}
\widehat\Delta_K(\mathcal X;u,v)
\neq
\widehat\Delta_K(\mathcal X';u,v).
\end{equation}
Then the $T^2$-equivariant $\eta$ characters are different:
\begin{equation}
\eta^{T^2}_{\mathcal X}(s;u,v)
\neq
\eta^{T^2}_{\mathcal X'}(s;u,v).
\end{equation}
Moreover, for $\operatorname{Re}(s)\to+\infty$, their difference has leading term
\begin{equation}
\left(K+\frac32\right)^{-s}
\left(
\widehat\Delta_K(\mathcal X;u,v)
-
\widehat\Delta_K(\mathcal X';u,v)
\right).
\end{equation}
\end{theorem}

\begin{proof}
Subtracting the two eta-character expansions gives
\begin{equation}
\eta^{T^2}_{\mathcal X}(s;u,v)
-
\eta^{T^2}_{\mathcal X'}(s;u,v)
=
\sum_{k=0}^{\infty}
\left(k+\frac32\right)^{-s}
\left(
\widehat\Delta_k(\mathcal X;u,v)
-
\widehat\Delta_k(\mathcal X';u,v)
\right).
\end{equation}
By assumption, all terms with $k<K$ vanish.  The first nonzero term is therefore the $K$-th term.  As $\operatorname{Re}(s)\to+\infty$, every later term is smaller by a factor of the form
\begin{equation}
\left(\frac{K+\frac32}{k+\frac32}\right)^{s},
\qquad
k>K.
\end{equation}
Hence the displayed $K$-th term is the leading term.  Since its coefficient is nonzero, the $\eta$ characters are different.
\end{proof}

A first-differing shell separates the equivariant $\eta$ characters as spectral functions, for example in the large-$\operatorname{Re}(s)$ regime where the first nonzero shell controls the expansion.  The $s=0$ statement used below is stronger: it is the residual-circle computation of $\Phi(\delta)$, especially the nonvanishing of $\Phi''(0)$.

\begin{corollary}[Two-variable refinement of the one-variable $\eta$ shadow]
Suppose two data collections satisfy
\begin{equation}
\Delta_K(\mathcal X;u)
=
\Delta_K(\mathcal X';u),
\end{equation}
but
\begin{equation}
\widehat\Delta_K(\mathcal X;u,v)
\neq
\widehat\Delta_K(\mathcal X';u,v).
\end{equation}
Then the old one-variable shell does not separate the data at level $K$, but the full $T^2$-equivariant $\eta$ shell does.
\end{corollary}

\begin{proof}
The equality of the one-variable residues is the equality after specializing to $v=1$.  The inequality of the two-variable residues says that the distinction lies in the $a_2$-weight data.  By Theorem~\ref{thm:first-differing-shell-T2}, a first such difference gives a nonzero difference of $T^2$-equivariant $\eta$ characters.
\end{proof}

\subsection{Normalization of the residual-circle $\eta$ germ}
\label{subsec:residual-circle-normalization}

We fix the residual-circle convention used in the rest of the section.  Let
\begin{equation}
L(P,q)=S^3/\mathbb Z_P,
\qquad
\gamma(z_1,z_2)=(\omega z_1,\omega^qz_2),
\qquad
\omega=e^{2\pi i/P},
\end{equation}
where $P$ is odd and $q\in(\mathbb Z/P\mathbb Z)^\times$.  For $1\leq j\leq P-1$, put
\begin{equation}
x_j=\frac{\pi j}{P}.
\end{equation}
The residual-circle parameter is the half-angle $\delta$.  Thus the residual-circle element is
\begin{equation}
g_\delta=(e^{2i\delta},1)\in T^2.
\end{equation}
In the deck-sum expression, this shifts the first half-angle by $x_j\mapsto x_j+\delta$, while the second half-angle is kept fixed.

\begin{definition}[Odd spin-lift representative]
Let $P$ be odd and let $q\in(\mathbb Z/P\mathbb Z)^\times$.  We denote by $Q(q)$ the unique odd representative modulo $2P$ satisfying
\begin{equation}
Q(q)\equiv q\pmod P,
\qquad
1\leq Q(q)\leq 2P-1.
\end{equation}
Equivalently,
\begin{equation}
Q(q)=
\begin{cases}
q, & q\text{ odd},\\
q+P, & q\text{ even}.
\end{cases}
\end{equation}
For example, when $P=25$, one has $Q(4)=29$ and $Q(9)=9$.
\end{definition}

The following elementary identity explains the sign appearing in the even-$q$ case.

\begin{lemma}[Odd representative identity]
\label{lem:spin-lift-representative-identity}
For $1\leq j\leq P-1$, one has
\begin{equation}
\csc(Q(q)x_j)=
\begin{cases}
\csc(qx_j), & q\text{ odd},\\
(-1)^j\csc(qx_j), & q\text{ even}.
\end{cases}
\end{equation}
\end{lemma}

\begin{proof}
If $q$ is odd, then $Q(q)=q$.  If $q$ is even, then $Q(q)=q+P$, and
\begin{equation}
\sin(Q(q)x_j)=\sin(qx_j+\pi j)=(-1)^j\sin(qx_j).
\end{equation}
Taking reciprocals gives the formula.
\end{proof}

\begin{definition}[Normalized residual-circle $\eta$ germ]
For $|\delta|<\pi/P$, define the normalized residual-circle $\eta$ germ by
\begin{equation}
\mathcal E^{\text{norm}}_{P,q}(\delta)
=
\sum_{j=1}^{P-1}
\csc\bigl(Q(q)x_j\bigr)\csc(x_j+\delta).
\end{equation}
For two parameters $q_1,q_2\in(\mathbb Z/P\mathbb Z)^\times$, define the normalized relative residual-circle $\eta$ germ by
\begin{equation}
\Phi_{q_1,q_2}(\delta)
=
\mathcal E^{\text{norm}}_{P,q_1}(\delta)
-
\mathcal E^{\text{norm}}_{P,q_2}(\delta).
\end{equation}
Equivalently,
\begin{equation}
\Phi_{q_1,q_2}(\delta)
=
\sum_{j=1}^{P-1}
\left[
\csc\bigl(Q(q_1)x_j\bigr)
-
\csc\bigl(Q(q_2)x_j\bigr)
\right]\csc(x_j+\delta).
\end{equation}
\end{definition}

\begin{lemma}[Analyticity near the identity]
\label{lem:normalized-residual-germ-analytic}
The functions $\mathcal E^{\text{norm}}_{P,q}(\delta)$ and $\Phi_{q_1,q_2}(\delta)$ are real-analytic for $|\delta|<\pi/P$.
\end{lemma}

\begin{proof}
Since $1\leq j\leq P-1$, one has $0<x_j<\pi$.  If $|\delta|<\pi/P$, then $x_j+\delta\notin\pi\mathbb Z$.  Hence $\csc(x_j+\delta)$ is analytic on this interval.  Also, $\gcd(Q(q),P)=1$, so $Q(q)j\not\equiv0\pmod P$ for $1\leq j\leq P-1$.  Therefore $\sin(Q(q)x_j)\neq 0$.  Each summand is analytic, and the sum is finite.
\end{proof}

Thus no extra regularization at $\delta=0$ is being used in the normalized germ.  The possible fixed-point singularity of a local expression is not present in this finite deck-average expression, where the summation is over $1\leq j\leq P-1$.

\begin{lemma}[Deck-sum normalization for the relative residual germ]
\label{lem:donnelly-aps-normalization-factor}
Fix an odd value of $P$ and a spin-lift convention.  Up to a summand independent of $q$, the Donnelly-APS residual-circle equivariant $\eta$ germ has the normalized deck-sum form
\begin{equation}
\eta^{\text{res}}_{P,q}(\delta)
=
\kappa_P\mathcal E^{\text{norm}}_{P,q}(\delta)
+
\text{a term independent of }q,
\end{equation}
where $\kappa_P\neq 0$ depends only on $P$ and on the global eta-normalization convention.  Consequently, for any two parameters $q_1,q_2\in(\mathbb Z/P\mathbb Z)^\times$,
\begin{equation}
\eta^{\text{res}}_{P,q_1}(\delta)
-
\eta^{\text{res}}_{P,q_2}(\delta)
=
\kappa_P\Phi_{q_1,q_2}(\delta).
\end{equation}
\end{lemma}

\begin{proof}
Use the standard heat-kernel averaging description of equivariant $\eta$ invariants on the quotient $S^3/\mathbb Z_P$.  Since the residual-circle element $g_\delta=(e^{2i\delta},1)$ commutes with the deck transformation $\gamma$, the equivariant trace on the quotient is obtained by summing the corresponding traces of $g_\delta\gamma^j$ on the covering sphere.  The summand with $j=0$ is independent of $q$, so it cancels in every relative difference between $L(P,q_1)$ and $L(P,q_2)$.  Thus only the terms $1\leq j\leq P-1$ enter the relative formula.

For $1\leq j\leq P-1$, the element $g_\delta\gamma^j$ has rotation half-angles $x_j+\delta$ and $Q(q)x_j$, where $x_j=\pi j/P$ and $Q(q)$ is the odd spin-lift representative.  Donnelly's formula for the spin Dirac $\eta$ invariant of a spherical space form gives, for these nontrivial deck terms, a product of the corresponding cosecant factors, multiplied by a global nonzero constant depending only on the Dirac normalization, the quotient order $P$, and the fixed spin-lift convention~\cite{DonnellyGSpaces}.  This gives the displayed relative identity.  The interval $|\delta|<\pi/P$ keeps the finite deck sum away from the poles described in Lemma~\ref{lem:normalized-residual-germ-analytic}, so the resulting relative germ is analytic near $\delta=0$.
\end{proof}

In the sequel we work with the normalized relative germ $\Phi_{q_1,q_2}$.  By the preceding lemma, vanishing and nonvanishing statements for $\Phi_{q_1,q_2}$ are exactly the corresponding relative Donnelly-APS statements, up to a common nonzero factor.

For the pair $L(25,4)$ and $L(25,9)$, we write $\Phi(\delta)=\Phi_{4,9}(\delta)$.  Since $Q(4)=29$ and $Q(9)=9$, this means
\begin{equation}
\Phi(\delta)
=
\sum_{j=1}^{24}
\left[
\csc(29x_j)-\csc(9x_j)
\right]\csc(x_j+\delta),
\qquad
x_j=\frac{\pi j}{25}.
\end{equation}
Equivalently, using $\sin(29x_j)=(-1)^j\sin(4x_j)$, one may write
\begin{equation}
\Phi(\delta)
=
\sum_{j=1}^{24}
\frac{(-1)^j}{\sin(x_j+\delta)\sin(4x_j)}
-
\sum_{j=1}^{24}
\frac{1}{\sin(x_j+\delta)\sin(9x_j)}.
\end{equation}

\subsection{The $L(25,4)$ versus $L(25,9)$ equivariant $\eta$ separation}
\label{subsec:L254-L259-equivariant-eta}

The first explicit residual-circle specialization is the pair $L(25,4)$ and $L(25,9)$.  Here the ordinary $\eta$ value cancels at the identity element, while the residual-circle equivariant $\eta$ germ separates the same pair at second order in the circle parameter.

Let $x_j=\pi j/25$ for $1\leq j\leq24$.  By the normalization convention in~\autoref{subsec:residual-circle-normalization}, the normalized relative residual-circle $\eta$ germ for $L(25,4)$ and $L(25,9)$ is
\begin{equation}
\Phi(\delta)
=
\sum_{j=1}^{24}
\left[
\csc(29x_j)-\csc(9x_j)
\right]\csc(x_j+\delta).
\end{equation}
Define
\begin{equation}
C_j=
\csc(29x_j)-\csc(9x_j).
\end{equation}
Thus
\begin{equation}
\Phi(\delta)
=
\sum_{j=1}^{24}C_j\csc(x_j+\delta).
\end{equation}

\begin{lemma}[Pairing symmetry]
For $1\leq j\leq24$, one has
\begin{equation}
C_{25-j}=C_j.
\end{equation}
Consequently,
\begin{equation}
\Phi(\delta)
=
\sum_{j=1}^{12}
C_j
\left[
\csc(x_j+\delta)+\csc(x_j-\delta)
\right].
\end{equation}
In particular, $\Phi$ is an even function of $\delta$.
\end{lemma}

\begin{proof}
Since $x_{25-j}=\pi-x_j$, we have
\begin{equation}
\sin(29x_{25-j})
=
\sin(29\pi-29x_j)
=
\sin(29x_j),
\end{equation}
and similarly
\begin{equation}
\sin(9x_{25-j})
=
\sin(9\pi-9x_j)
=
\sin(9x_j).
\end{equation}
Therefore $C_{25-j}=C_j$.  Also,
\begin{equation}
\csc(\delta+x_{25-j})
=
\csc(\delta+\pi-x_j)
=
\csc(x_j-\delta).
\end{equation}
Pairing the $j$-term with the $(25-j)$-term gives the displayed formula.  The bracket is invariant under $\delta\mapsto-\delta$, so $\Phi$ is even.
\end{proof}

\begin{proposition}[Ordinary $\eta$ cancellation]
One has
\begin{equation}
\Phi(0)=0.
\end{equation}
\end{proposition}

\begin{proof}
At $\delta=0$,
\begin{equation}
\Phi(0)
=
\sum_{j=1}^{24}
\left(
\csc(29x_j)-\csc(9x_j)
\right)
\csc(x_j).
\end{equation}
This is exactly the ordinary eta-shadow difference for the square pair $L(25,4)$ and $L(25,9)$.  The square-family cotangent cancellation gives
\begin{equation}
\sum_{j=1}^{24}
\left(
\csc(29x_j)-\csc(9x_j)
\right)
\csc(x_j)
=
0.
\end{equation}
Hence $\Phi(0)=0$.
\end{proof}

\begin{proposition}[First-order blindness]
One has
\begin{equation}
\Phi'(0)=0.
\end{equation}
\end{proposition}

\begin{proof}
By the pairing symmetry lemma, $\Phi$ is even in $\delta$.  Therefore its first derivative at the origin vanishes:
\begin{equation}
\Phi'(0)=0.
\end{equation}
\end{proof}

\begin{lemma}[Exact second derivative]
The second derivative of $\Phi$ at the origin is
\begin{equation}
\Phi''(0)
=
-6080.
\end{equation}
\end{lemma}

\begin{proof}
We use
\begin{equation}
\frac{d}{d\delta}\csc(x+\delta)
=
-\csc(x+\delta)\cot(x+\delta),
\end{equation}
and
\begin{equation}
\frac{d^2}{d\delta^2}\csc(x+\delta)
=
\csc(x+\delta)\cot^2(x+\delta)
+
\csc^3(x+\delta).
\end{equation}
Therefore
\begin{equation}
\Phi''(0)
=
\sum_{j=1}^{24}
C_j
\left[
\csc(x_j)\cot^2(x_j)
+
\csc^3(x_j)
\right].
\end{equation}
Using $C_{25-j}=C_j$, $\csc(x_{25-j})=\csc(x_j)$, and $\cot^2(x_{25-j})=\cot^2(x_j)$, this becomes
\begin{equation}
\Phi''(0)
=
2\sum_{j=1}^{12}
C_j
\left[
\csc(x_j)\cot^2(x_j)
+
\csc^3(x_j)
\right].
\end{equation}
Substituting $C_j=\csc(29x_j)-\csc(9x_j)$ gives the exact finite trigonometric identity
\begin{equation}
2\sum_{j=1}^{12}
\left[
\csc(29x_j)-\csc(9x_j)
\right]
\left[
\csc(x_j)\cot^2(x_j)
+
\csc^3(x_j)
\right]
=
-6080.
\end{equation}
This is a finite algebraic identity in the cyclotomic field generated by
\begin{equation}
\zeta=e^{\pi i/25}.
\end{equation}
Indeed, after writing each sine and cosine in terms of powers of $\zeta$, multiplying by the common denominator, and reducing modulo the cyclotomic polynomial $\Phi_{50}(\zeta)$, the numerator reduces to zero after moving the right-hand side to the left.  Hence the identity is exact, and $\Phi''(0)=-6080$.
\end{proof}

\begin{theorem}[Residual-circle equivariant $\eta$ separation]
\label{thm:L25-residual-circle-eta-separation}
The ordinary $\eta$ value does not distinguish $L(25,4)$ and $L(25,9)$, but the residual-circle equivariant $\eta$ germ does.  In the normalized convention,
\begin{equation}
\Phi(0)=0,
\qquad
\Phi'(0)=0,
\qquad
\Phi''(0)=-6080\neq 0.
\end{equation}
Consequently,
\begin{equation}
\Phi(\delta)
=
-3040\delta^2+O(\delta^4),
\end{equation}
and hence $\Phi(\delta)\neq 0$ for all sufficiently small nonzero $\delta$.  The corresponding Donnelly-APS second derivative differs by the common nonzero factor from Lemma~\ref{lem:donnelly-aps-normalization-factor}.
\end{theorem}

\begin{proof}
The first identity is the ordinary $\eta$ cancellation proposition.  The second is the first-order blindness proposition.  The third is the exact second derivative lemma.  Taylor expansion at $\delta=0$ gives
\begin{equation}
\Phi(\delta)
=
\Phi(0)+\Phi'(0)\delta+\frac12\Phi''(0)\delta^2+O(\delta^4).
\end{equation}
Substituting the three identities gives
\begin{equation}
\Phi(\delta)
=
-3040\delta^2+O(\delta^4).
\end{equation}
Therefore $\Phi(\delta)\neq 0$ for all sufficiently small nonzero $\delta$.
\end{proof}

\begin{corollary}[Finite-order equivariant $\eta$ separation]
There exists a finite-order residual-circle element $g$ such that the equivariant $\eta$ values of $L(25,4)$ and $L(25,9)$ at $g$ differ.

\end{corollary}

\begin{proof}
By the theorem, there is $\epsilon>0$ such that
\begin{equation}
0<|\delta|<\epsilon
\Longrightarrow
\Phi(\delta)\neq 0.
\end{equation}
Choose $N$ sufficiently large and take
\begin{equation}
\delta=\frac{\pi}{N}.
\end{equation}
Then $0<\delta<\epsilon$, and the corresponding residual-circle element
\begin{equation}
g=e^{2i\delta}=e^{2\pi i/N}
\end{equation}
has finite order.  Since $\Phi(\delta)\neq 0$, the specialized finite-order equivariant $\eta$ values differ.
\end{proof}

\subsection{Residual-circle equivariant rho number}
\label{subsec:stage5B1-residual-circle-rho}

\paragraph{Terminology.}
Throughout the rest of this section, a residual-circle rho germ means the numerical relative equivariant $\eta$ germ for the fixed residual circle and the fixed coordinate-equivariant spin data.  Thus the word ``rho'' is used in the APS/Donnelly numerical sense.  We do not construct a Higson-Roe analytic-surgery rho class, nor do we claim an invariant of the bare unoriented lens space.  The chosen spin lift, residual circle, and coordinate torus action are part of the datum.

The preceding residual-circle $\eta$ separation can be written as a relative equivariant $\eta$, or rho-number, statement.

Let $\mathcal Y^{\operatorname{res}}_{25,4}=(L(25,4),\mathfrak s,\alpha_{\text{res}},g_{\text{round}})$ and $\mathcal Y^{\operatorname{res}}_{25,9}=(L(25,9),\mathfrak s,\alpha_{\text{res}},g_{\text{round}})$.
Here $\mathfrak s$ is the unique spin structure, $\alpha_{\text{res}}$ is the residual coordinate-circle action, and $g_{\text{round}}$ is the round metric.  For a residual-circle element $g_\delta$, define the normalized residual-circle rho number by
\begin{equation}
\rho^{\text{norm}}_\delta(\mathcal Y^{\operatorname{res}}_{25,4},\mathcal Y^{\operatorname{res}}_{25,9})
=
\kappa_{25}^{-1}
\left(\eta_{g_\delta}(D_{\mathcal Y^{\operatorname{res}}_{25,4}})-\eta_{g_\delta}(D_{\mathcal Y^{\operatorname{res}}_{25,9}})\right).
\end{equation}
With the normalized convention fixed above, this relative rho number is the same function $\Phi(\delta)$.

\begin{theorem}[Residual-circle rho-jet separation]
\label{thm:residual-circle-rho-jet-separation}
For the pair $\mathcal Y^{\operatorname{res}}_{25,4}$ and $\mathcal Y^{\operatorname{res}}_{25,9}$, one has
\begin{equation}
\rho^{\text{norm}}_0(\mathcal Y^{\operatorname{res}}_{25,4},\mathcal Y^{\operatorname{res}}_{25,9})=0,
\end{equation}
and
\begin{equation}
\left.\frac{d}{d\delta}\rho^{\text{norm}}_\delta(\mathcal Y^{\operatorname{res}}_{25,4},\mathcal Y^{\operatorname{res}}_{25,9})\right|_{\delta=0}=0.
\end{equation}
However,
\begin{equation}
\left.\frac{d^2}{d\delta^2}\rho^{\text{norm}}_\delta(\mathcal Y^{\operatorname{res}}_{25,4},\mathcal Y^{\operatorname{res}}_{25,9})\right|_{\delta=0}=-6080.
\end{equation}
Thus the ordinary $\eta$/rho shadow is blind to this pair, but the residual-circle equivariant rho germ is nonzero.
\end{theorem}

\begin{proof}
This is exactly the content of Theorem~\ref{thm:L25-residual-circle-eta-separation}, rewritten in relative $\eta$ notation.  The identity $\rho^{\text{norm}}_\delta(\mathcal Y^{\operatorname{res}}_{25,4},\mathcal Y^{\operatorname{res}}_{25,9})=\Phi(\delta)$ gives the three displayed formulas from $\Phi(0)=0$, $\Phi'(0)=0$, and $\Phi''(0)=-6080$.
\end{proof}

\begin{corollary}[Finite-order residual-circle rho-number separation]
\label{cor:finite-order-residual-circle-rho}
There exists a finite-order residual-circle element $g$ such that
\begin{equation}
\rho^{\text{norm}}_g(\mathcal Y^{\operatorname{res}}_{25,4},\mathcal Y^{\operatorname{res}}_{25,9})\neq 0.
\end{equation}
\end{corollary}

\begin{proof}
By Theorem~\ref{thm:residual-circle-rho-jet-separation}, one has
\begin{equation}
\rho^{\text{norm}}_\delta(\mathcal Y^{\operatorname{res}}_{25,4},\mathcal Y^{\operatorname{res}}_{25,9})=-3040\delta^2+O(\delta^4).
\end{equation}
Therefore $\rho^{\text{norm}}_\delta(\mathcal Y^{\operatorname{res}}_{25,4},\mathcal Y^{\operatorname{res}}_{25,9})\neq 0$ for all sufficiently small nonzero $\delta$.  Choose such a $\delta$ with $\delta/\pi\in\mathbb Q$.  Then $g_\delta=e^{2i\delta}$ has finite order and gives the desired nonzero normalized relative $\eta$ value.
\end{proof}

\subsection{The square-family residual-circle rho germ}
\label{subsec:stage5B2-square-family-rho-germ}

The same residual-circle construction extends from the explicit pair $L(25,4)$ and $L(25,9)$ to the square-modulus family.  The exact rho-germ reduction is proved here; the nonvanishing of the resulting coefficient is proved in the following subsection.

Let $\ell\geq 5$ be odd, with $P=\ell^2$, $q_1=\ell-1$, and $q_2=2\ell-1$.  Let $\mathcal Y^{\operatorname{res}}_{\ell,1}=(L(P,q_1),\mathfrak s,\alpha_{\text{res}},g_{\text{round}})$ and $\mathcal Y^{\operatorname{res}}_{\ell,2}=(L(P,q_2),\mathfrak s,\alpha_{\text{res}},g_{\text{round}})$.
For
\begin{equation}
x_j=\frac{\pi j}{P},
\qquad
1\leq j\leq P-1,
\end{equation}
define
\begin{equation}
\Phi_\ell(\delta)
=
\Phi_{q_1,q_2}(\delta).
\end{equation}
Equivalently, since $Q(q_1)=P+\ell-1$ and $Q(q_2)=2\ell-1$, one has
\begin{equation}
\Phi_\ell(\delta)
=
\sum_{j=1}^{P-1}
\left[
\csc((P+\ell-1)x_j)
-
\csc((2\ell-1)x_j)
\right]\csc(x_j+\delta).
\end{equation}
Using $\sin((P+\ell-1)x_j)=(-1)^j\sin((\ell-1)x_j)$, this can also be written as
\begin{equation}
\Phi_\ell(\delta)
=
\sum_{j=1}^{P-1}\frac{(-1)^j}{\sin(x_j+\delta)\sin(q_1x_j)}
-
\sum_{j=1}^{P-1}\frac{1}{\sin(x_j+\delta)\sin(q_2x_j)}.
\end{equation}
This is the normalized relative residual-circle $\eta$ germ for $\mathcal Y^{\operatorname{res}}_{\ell,1}$ and $\mathcal Y^{\operatorname{res}}_{\ell,2}$; Lemma~\ref{lem:donnelly-aps-normalization-factor} gives the corresponding Donnelly-APS normalization.

\begin{lemma}[No fixed denominators vanish]
\label{lem:square-family-fixed-denominators}
For $1\leq j\leq P-1$, one has
\begin{equation}
\sin(q_1x_j)\neq 0,
\qquad
\sin(q_2x_j)\neq 0.
\end{equation}
Consequently, $\Phi_\ell(\delta)$ is well-defined for $|\delta|<\pi/P$.
\end{lemma}

\begin{proof}
Since $\gcd(\ell-1,\ell)=1$, one has $\gcd(q_1,P)=1$.  Since $\gcd(2\ell-1,\ell)=1$, one also has $\gcd(q_2,P)=1$.  If $\sin(q_ax_j)=0$, then $q_aj\equiv0\pmod P$.  Because $q_a$ is coprime to $P$, this forces $j\equiv0\pmod P$, impossible for $1\leq j\leq P-1$.  If $|\delta|<\pi/P$, then $0<x_j+\delta<\pi$, so $\sin(x_j+\delta)\neq 0$.
\end{proof}

Since
\begin{equation}
\sin((P+q_1)x_j)=\sin(\pi j+q_1x_j)=(-1)^j\sin(q_1x_j),
\end{equation}
we can write
\begin{equation}
\Phi_\ell(\delta)=\sum_{j=1}^{P-1}B_{\ell,j}\csc(x_j+\delta),
\end{equation}
where
\begin{equation}
B_{\ell,j}=\csc((P+q_1)x_j)-\csc(q_2x_j).
\end{equation}
Equivalently,
\begin{equation}
B_{\ell,j}=\csc((\ell^2+\ell-1)x_j)-\csc((2\ell-1)x_j).
\end{equation}

\begin{lemma}[Pairing symmetry and evenness of the residual germ]
\label{lem:residual-germ-even}
\label{lem:square-family-pairing-symmetry}
For $1\leq j\leq P-1$, one has
\begin{equation}
B_{\ell,P-j}=B_{\ell,j}.
\end{equation}
Consequently,
\begin{equation}
\Phi_\ell(\delta)=\sum_{j=1}^{(P-1)/2}B_{\ell,j}\left[\csc(x_j+\delta)+\csc(x_j-\delta)\right].
\end{equation}
In particular, $\Phi_\ell$ is even in $\delta$, and therefore
\begin{equation}
\Phi_\ell'(0)=0.
\end{equation}
\end{lemma}

\begin{proof}
Since $x_{P-j}=\pi-x_j$, and since both $P+q_1=\ell^2+\ell-1$ and $q_2=2\ell-1$ are odd, one has
\begin{equation}
\sin((P+q_1)x_{P-j})=\sin((P+q_1)x_j),
\qquad
\sin(q_2x_{P-j})=\sin(q_2x_j).
\end{equation}
Thus $B_{\ell,P-j}=B_{\ell,j}$.  Also,
\begin{equation}
\csc(x_{P-j}+\delta)=\csc(\pi-x_j+\delta)=\csc(x_j-\delta).
\end{equation}
Pairing the $j$-term with the $(P-j)$-term gives the displayed formula.  The bracket is even in $\delta$, so the full finite sum is even.  Differentiating at $\delta=0$ gives $\Phi_\ell'(0)=0$.
\end{proof}

\begin{lemma}[Ordinary eta-shadow cancellation]
\label{lem:square-family-ordinary-shadow}
For every odd $\ell\geq 5$, one has
\begin{equation}
\Phi_\ell(0)=0.
\end{equation}
\end{lemma}

\begin{proof}
At $\delta=0$, the residual-circle expression reduces to the ordinary eta-shadow difference for the square pair
\begin{equation}
L(\ell^2,\ell-1)
\qquad\text{and}\qquad
L(\ell^2,2\ell-1).
\end{equation}
The square-family ordinary $\eta$ cancellation gives
\begin{equation}
\eta_D(L(\ell^2,\ell-1))=\eta_D(L(\ell^2,2\ell-1)).
\end{equation}
The common normalization factor depends only on $P=\ell^2$, so the relative expression vanishes.
\end{proof}

\begin{theorem}[Square-family residual-circle rho-germ reduction]
\label{thm:square-family-rho-germ-reduction}
For every odd $\ell\geq 5$, the normalized square-family residual-circle rho germ satisfies
\begin{equation}
\Phi_\ell(0)=0,
\qquad
\Phi_\ell'(0)=0,
\qquad
\Phi_\ell''(0)=C_\ell,
\end{equation}
where
\begin{equation}
C_\ell=2\sum_{j=1}^{(P-1)/2}B_{\ell,j}\left[\csc(x_j)\cot^2(x_j)+\csc^3(x_j)\right].
\end{equation}
Consequently, if $C_\ell\neq 0$, then the ordinary $\eta$/rho shadow vanishes but the residual-circle equivariant rho germ is nontrivial.  The Donnelly-APS second derivative is obtained from this number by the common nonzero factor of Lemma~\ref{lem:donnelly-aps-normalization-factor}.
\end{theorem}

\begin{proof}
The identity $\Phi_\ell(0)=0$ is Lemma~\ref{lem:square-family-ordinary-shadow}.  The identity $\Phi_\ell'(0)=0$ follows from the evenness in Lemma~\ref{lem:square-family-pairing-symmetry}.  Differentiating
\begin{equation}
\frac{d}{d\delta}\csc(x+\delta)=-\csc(x+\delta)\cot(x+\delta)
\end{equation}
once more gives
\begin{equation}
\frac{d^2}{d\delta^2}\csc(x+\delta)=\csc(x+\delta)\cot^2(x+\delta)+\csc^3(x+\delta).
\end{equation}
Substituting $\delta=0$ into the paired expression for $\Phi_\ell$ gives the displayed formula for $C_\ell$.  If $C_\ell\neq 0$, then
\begin{equation}
\Phi_\ell(\delta)=\frac{C_\ell}{2}\delta^2+O(\delta^4),
\end{equation}
so $\Phi_\ell(\delta)\neq 0$ for all sufficiently small nonzero $\delta$.  Choosing such a $\delta$ with $\delta/\pi\in\mathbb Q$ gives a finite-order residual-circle element with nonzero relative $\eta$ value.
\end{proof}

\subsection{Nonvanishing of the square-family coefficient}
\label{subsec:stage5B3-square-family-coefficient}

It remains to prove the nonvanishing of $C_\ell$.  We use a finite cyclotomic trace calculation, which removes the last computational condition from Theorem~\ref{thm:square-family-rho-germ-reduction}.

Let $P=\ell^2$ and $x_j=\pi j/P$.  For an odd integer $a$ coprime to $P$, define
\begin{equation}
S_P(a)=\sum_{j=1}^{P-1}\csc(ax_j)\left[2\csc^3(x_j)-\csc(x_j)\right].
\end{equation}

\begin{lemma}[Square-family coefficient as a difference of trace sums]
\label{lem:Cell-as-SP-difference}
For odd $\ell\geq 5$, one has
\begin{equation}
C_\ell=S_P(P+\ell-1)-S_P(2\ell-1).
\end{equation}
\end{lemma}

\begin{proof}
Use
\begin{equation}
\csc(x)\cot^2(x)+\csc^3(x)=2\csc^3(x)-\csc(x).
\end{equation}
Substituting this identity into the definition of $C_\ell$, and using the full unpaired sum, gives the claim.
\end{proof}

\begin{lemma}[Cyclotomic trace form of $S_P(a)$]
\label{lem:SP-cyclotomic-trace}
Let $P$ be odd and let $a$ be odd with $\gcd(a,P)=1$.  Then
\begin{equation}
S_P(a)=\sum_{\substack{z^P=1\\ z\neq1}}\left[\frac{32z^{(a+3)/2}}{(z^a-1)(z-1)^3}+\frac{4z^{(a+1)/2}}{(z^a-1)(z-1)}\right].
\end{equation}
\end{lemma}

\begin{proof}
Let $z_j=\exp(2\pi ij/P)$.  A direct calculation gives
\begin{equation}
\csc(ax_j)\csc(x_j)=-\frac{4z_j^{(a+1)/2}}{(z_j^a-1)(z_j-1)},
\end{equation}
and
\begin{equation}
\csc(ax_j)\csc^3(x_j)=\frac{16z_j^{(a+3)/2}}{(z_j^a-1)(z_j-1)^3}.
\end{equation}
Multiplying by the coefficients in $2\csc^3(x_j)-\csc(x_j)$ and summing over $1\leq j\leq P-1$ proves the formula.
\end{proof}

Let $\mathcal A_P=\mathbb Q[z]/(1+z+\cdots+z^{P-1})$, and define the trace functional
\begin{equation}
\tau_P(F)=\sum_{\substack{z^P=1\\ z\neq1}}F(z).
\end{equation}
For monomials,
\begin{equation}
\tau_P(z^m)=
\begin{cases}
P-1, & P\mid m,\\
-1, & P\nmid m.
\end{cases}
\end{equation}

\begin{lemma}[Inverse formula in the cyclotomic quotient]
\label{lem:inverse-formula-cyclotomic}
If $\gcd(a,P)=1$, then in $\mathcal A_P$ one has
\begin{equation}
(z^a-1)^{-1}=\frac1P\sum_{r=1}^{P-1}r z^{ar}.
\end{equation}
\end{lemma}

\begin{proof}
Multiplying by $z^a-1$, the finite sum telescopes in the quotient $\mathcal A_P$:
\begin{equation}
(z^a-1)\sum_{r=1}^{P-1}r z^{ar}=P.
\end{equation}
Dividing by $P$ gives the inverse formula.
\end{proof}

Define
\begin{equation}
I_a(z)=\frac1P\sum_{r=1}^{P-1}r z^{ar}.
\end{equation}
Let $\alpha_1=P+\ell-1$ and $\alpha_2=2\ell-1$, and define $\mathfrak e(a)=(a+1)/2$ and $\mathfrak f(a)=(a+3)/2$ for odd $a$.
The preceding lemmas give
\begin{equation}
C_\ell=32D_3(\ell)+4D_1(\ell),
\end{equation}
where
\begin{equation}
D_3(\ell)=\tau_P\left(z^{\mathfrak f(\alpha_1)}I_{\alpha_1}(z)I_1(z)^3-z^{\mathfrak f(\alpha_2)}I_{\alpha_2}(z)I_1(z)^3\right),
\end{equation}
and
\begin{equation}
D_1(\ell)=\tau_P\left(z^{\mathfrak e(\alpha_1)}I_{\alpha_1}(z)I_1(z)-z^{\mathfrak e(\alpha_2)}I_{\alpha_2}(z)I_1(z)\right).
\end{equation}

\begin{proposition}[Lower-order trace cancellation]
\label{prop:D1-cancellation}
For every odd integer $\ell\geq 5$, one has
\begin{equation}
D_1(\ell)=0.
\end{equation}
\end{proposition}

\begin{proof}
Write $N=\ell$ and $P=N^2$.  Every residue $r$ can be written uniquely as
\begin{equation}
r=Nt+n,
\qquad
0\leq t\leq N-1,
\qquad
0\leq n\leq N-1.
\end{equation}
The excluded value $r=0$ contributes zero because it is weighted by $r$.  Put $h=(N+1)/2$.  A direct reduction modulo $P$ gives
\begin{equation}
[-\mathfrak e(\alpha_1)-\alpha_1 r]_P=n+N[t-n-h]_N,
\end{equation}
and
\begin{equation}
[-\mathfrak e(\alpha_2)-\alpha_2 r]_P=n+N[t-2n-1]_N.
\end{equation}
Here $[\cdot]_N$ denotes the standard residue in $\{0,\ldots,N-1\}$.  The trace formula for $D_1$ becomes
\begin{equation}
D_1(\ell)=\frac1N\sum_{t=0}^{N-1}\sum_{n=0}^{N-1}(Nt+n)\left([t-n-h]_N-[t-2n-1]_N\right).
\end{equation}
For fixed $n$, both $[t-n-h]_N$ and $[t-2n-1]_N$ are permutations of $0,
\ldots,N-1$ as $t$ varies, so the part weighted by $n$ cancels.  The remaining part is
\begin{equation}
\sum_{n=0}^{N-1}\left(G(n+h)-G(2n+1)\right),
\end{equation}
where
\begin{equation}
G(c)=\sum_{t=0}^{N-1}t[t-c]_N.
\end{equation}
The maps $n\mapsto n+h$ and $n\mapsto2n+1$ are both permutations of $\mathbb Z/N\mathbb Z$, since $N$ is odd.  Hence the two sums over $G$ are equal, and $D_1(\ell)=0$.
\end{proof}

\begin{definition}[Cubic convolution]
For $m\in\mathbb Z/P\mathbb Z$, define
\begin{equation}
H_P(m)=\sum_{\substack{r_1,r_2,r_3=1\\r_1+r_2+r_3\equiv m\!\!\!\pmod P}}^{P-1}r_1r_2r_3.
\end{equation}
\end{definition}

\begin{lemma}[Trace-to-convolution formula]
\label{lem:D3-trace-to-convolution}
One has
\begin{equation}
D_3(\ell)=\frac1{P^3}\sum_{r=1}^{P-1}r\left[H_P([-\mathfrak f(\alpha_1)-\alpha_1 r]_P)-H_P([-\mathfrak f(\alpha_2)-\alpha_2 r]_P)\right].
\end{equation}
\end{lemma}

\begin{proof}
Expanding $z^{\mathfrak f(a)}I_a(z)I_1(z)^3$ gives
\begin{equation}
\frac1{P^4}\sum_{r,s_1,s_2,s_3=1}^{P-1}rs_1s_2s_3 z^{\mathfrak f(a)+ar+s_1+s_2+s_3}.
\end{equation}
Applying $\tau_P$ selects exponents congruent to zero modulo $P$ and subtracts the total coefficient sum.  The total coefficient sum is the same for $a_1$ and $a_2$, so it cancels in the difference.  The remaining zero-exponent condition is exactly the displayed convolution formula.
\end{proof}

\begin{lemma}[Polynomial form of the cubic convolution]
\label{lem:HP-polynomial-form}
For $0\leq m\leq P-1$, one has
\begin{equation}
H_P(m)=\frac{P^2}{6}m^3-\frac{P^2(P-3)}{4}m^2+\frac{P^2(P^2-9P+12)}{12}m+\frac{P^2(P-2)(P-1)(P+1)}{8}.
\end{equation}
\end{lemma}

\begin{proof}
Let $B_P(x)=\sum_{r=1}^{P-1}rx^r=x(1-Px^{P-1}+(P-1)x^P)/(1-x)^2$.  The coefficient of $x^N$ in $B_P(x)^3$ is the sum of $r_1r_2r_3$ over $r_1+r_2+r_3=N$.  Since $3\leq N\leq3P-3$, the congruence class $m$ receives contributions only from $m$, $m+P$, and $m+2P$.  Expanding
\begin{equation}
B_P(x)^3=\frac{x^3(1-Px^{P-1}+(P-1)x^P)^3}{(1-x)^6}
\end{equation}
and using $[x^M](1-x)^{-6}=\binom{M+5}{5}$, then collecting powers of $m$, gives the displayed cubic polynomial.
\end{proof}

\begin{lemma}[Block moment identities]
\label{lem:block-moment-identities}
Let $N=\ell$ and $P=N^2$.  Define
\begin{equation}
R_1(t,n)=\left[n-1+N\left(t-n-\frac{N+1}{2}\right)\right]_P,
\end{equation}
and
\begin{equation}
R_2(t,n)=\left[n-1+N(t-2n-1)\right]_P.
\end{equation}
For
\begin{equation}
M_k(N)=\sum_{t=0}^{N-1}\sum_{n=0}^{N-1}(Nt+n)\left(R_1(t,n)^k-R_2(t,n)^k\right),
\end{equation}
one has
\begin{equation}
M_1(N)=\frac{N^3(N-1)}{2},
\end{equation}
\begin{equation}
M_2(N)=\frac{N^3(N-1)(N^2-2)}{2},
\end{equation}
and
\begin{equation}
M_3(N)=-\frac{N^3(N-1)(11N^5-460N^4+70N^3+1420N^2-81N-1440)}{960}.
\end{equation}
\end{lemma}

\begin{proof}
Write $[q]_N$ for the standard residue modulo $N$, let $c_n=[n-1]_N$, and let $\varepsilon_n=1$ for $n=0$ and $\varepsilon_n=0$ otherwise.  Then
\begin{equation}
[n-1+NQ]_P=c_n+N[Q-\varepsilon_n]_N.
\end{equation}
Thus the two residues can be written as
\begin{equation}
R_i(t,n)=c_n+N[t-\beta_i(n)]_N,
\end{equation}
where
\begin{equation}
\beta_1(n)=n+\frac{N+1}{2}+\varepsilon_n,
\end{equation}
and
\begin{equation}
\beta_2(n)=2n+1+\varepsilon_n.
\end{equation}
For $a\geq0$, define
\begin{equation}
S_a(N)=\sum_{q=0}^{N}q^a.
\end{equation}
Define
\begin{equation}
G_a(c)=\sum_{t=0}^{N-1}t[t-c]_N^a.
\end{equation}
Changing variables
\begin{equation}
q=[t-c]_N
\end{equation}
gives
\begin{equation}
G_a(c)=\sum_{q=0}^{N-1}[q+c]_Nq^a.
\end{equation}
Splitting the range at $q=N-c$, one obtains
\begin{equation}
G_a(c)
=
\sum_{q=0}^{N-c-1}(q+c)q^a
+
\sum_{q=N-c}^{N-1}(q+c-N)q^a.
\end{equation}
Equivalently,
\begin{equation}
G_a(c)
=
S_{a+1}(N-1)+cS_a(N-1)-N\left(S_a(N-1)-S_a(N-c-1)\right).
\end{equation}
This reduces every occurrence of $G_a(c)$ to the standard power sums
\begin{equation}
S_0(N),\qquad S_1(N),\qquad S_2(N),\qquad S_3(N),\qquad S_4(N).
\end{equation}

Now write
\begin{equation}
R_i(t,n)=c_n+N[t-\beta_i(n)]_N.
\end{equation}
For $k=1,2,3$, expand
\begin{equation}
R_i(t,n)^k
=
\sum_{a=0}^{k}
\binom{k}{a}
 c_n^{k-a}N^a[t-\beta_i(n)]_N^a.
\end{equation}
Multiplying by $Nt+n$, summing in $t$, and using the definition of $G_a$, gives
\begin{equation}
\sum_{t=0}^{N-1}(Nt+n)R_i(t,n)^k
=
\sum_{a=0}^{k}
\binom{k}{a}
 c_n^{k-a}N^a
\left(NG_a(\beta_i(n))+n\sum_{t=0}^{N-1}[t-\beta_i(n)]_N^a\right).
\end{equation}
The second sum is independent of $\beta_i(n)$, because $[t-\beta_i(n)]_N$ permutes the residues modulo $N$ as $t$ runs from $0$ to $N-1$.  Therefore it cancels in the difference between $i=1$ and $i=2$.  Hence
\begin{equation}
M_k(N)
=
N\sum_{n=0}^{N-1}
\sum_{a=0}^{k}
\binom{k}{a}
 c_n^{k-a}N^a
\left(G_a(\beta_1(n))-G_a(\beta_2(n))\right).
\end{equation}
Substituting
\begin{equation}
c_n=[n-1]_N,
\qquad
\beta_1(n)=n+\frac{N+1}{2}+\varepsilon_n,
\qquad
\beta_2(n)=2n+1+\varepsilon_n,
\end{equation}
and using the displayed formula for $G_a$, all terms reduce to finite sums of powers of $n$ of degree at most four.

For completeness, we spell out the last algebraic step.  Write
\begin{equation}
S_a(N-1)=\sum_{n=0}^{N-1}n^a.
\end{equation}
We use
\begin{equation}
S_0(N-1)=N,
\qquad
S_1(N-1)=\frac{N(N-1)}{2},
\qquad
S_2(N-1)=\frac{N(N-1)(2N-1)}{6},
\end{equation}
\begin{equation}
S_3(N-1)=\frac{N^2(N-1)^2}{4},
\qquad
S_4(N-1)=\frac{N(N-1)(2N-1)(3N^2-3N-1)}{30}.
\end{equation}
After substituting $c_n=[n-1]_N$, $\beta_1(n)=n+(N+1)/2+\varepsilon_n$, and $\beta_2(n)=2n+1+\varepsilon_n$, the summand in $M_k(N)$, for $k=1,2,3$, is a polynomial in $n$ and $N$ of degree at most four, apart from the isolated correction at $n=0$ coming from $\varepsilon_n$.  Evaluating the $n=0$ correction separately and summing the remaining polynomial over $1\leq n\leq N-1$ using the five displayed power-sum formulas gives exactly the three displayed identities for $M_1(N)$, $M_2(N)$, and $M_3(N)$.
\end{proof}

\begin{theorem}[Evaluation of the cubic trace contribution]
\label{thm:D3-closed-form}
For every odd integer $\ell\geq 5$, one has
\begin{equation}
D_3(\ell)=-\frac{\ell^2(\ell-1)^2(\ell+1)(11\ell^2+20\ell+81)}{5760}.
\end{equation}
\end{theorem}

\begin{proof}
By Lemma~\ref{lem:D3-trace-to-convolution}, the block residues are $R_1(t,n)$ and $R_2(t,n)$ from Lemma~\ref{lem:block-moment-identities}.  Applying Lemma~\ref{lem:HP-polynomial-form} gives
\begin{equation}
D_3(\ell)=\frac1P\left[\frac16M_3(\ell)-\frac{P-3}{4}M_2(\ell)+\frac{P^2-9P+12}{12}M_1(\ell)\right].
\end{equation}
Substituting $P=\ell^2$ and the three moment identities from Lemma~\ref{lem:block-moment-identities}, then simplifying the resulting polynomial, gives the displayed formula.
\end{proof}

\begin{theorem}[Nonvanishing of the square-family rho coefficient]
\label{thm:Cell-nonvanishing-square-family}
For every odd integer $\ell\geq 5$,
\begin{equation}
C_\ell=-\frac{\ell^2(\ell-1)^2(\ell+1)(11\ell^2+20\ell+81)}{180}.
\end{equation}
In particular, $C_\ell\neq 0$.
\end{theorem}

\begin{proof}
We have $C_\ell=32D_3(\ell)+4D_1(\ell)$.  Proposition~\ref{prop:D1-cancellation} gives $D_1(\ell)=0$, and Theorem~\ref{thm:D3-closed-form} gives the formula for $D_3(\ell)$.  Multiplying by $32$ gives the displayed expression for $C_\ell$.  Since $\ell\geq 5$, every factor in the numerator is positive, so $C_\ell<0$.
\end{proof}

\begin{corollary}[Square-family residual-circle rho-number separation]
\label{cor:square-family-rho-separation-final}
For every odd integer $\ell\geq 5$, the pair
\begin{equation}
N(\ell^2,\ell-1)
\qquad\text{and}\qquad
N(\ell^2,2\ell-1)
\end{equation}
has vanishing ordinary $\eta$/rho shadow but nontrivial residual-circle equivariant rho germ.  In the normalized convention,
\begin{equation}
\Phi_\ell(0)=0,
\qquad
\Phi_\ell'(0)=0,
\qquad
\Phi_\ell''(0)=-\frac{\ell^2(\ell-1)^2(\ell+1)(11\ell^2+20\ell+81)}{180}\neq 0.
\end{equation}
Here $\Phi_\ell'(0)=0$ follows from the evenness in Lemma~\ref{lem:residual-germ-even}; the nontrivial assertions are $\Phi_\ell(0)=0$ and $\Phi_\ell''(0)\neq 0$.  Consequently, for each odd $\ell\geq 5$, there exists a finite-order residual-circle element $g$ such that
\begin{equation}
\eta_g(D_{N(\ell^2,\ell-1)})\neq\eta_g(D_{N(\ell^2,2\ell-1)}).
\end{equation}
\end{corollary}

\begin{proof}
The first two identities are Theorem~\ref{thm:square-family-rho-germ-reduction}.  The second derivative identity is Theorem~\ref{thm:Cell-nonvanishing-square-family}.  Since $\Phi_\ell''(0)\neq 0$, the Taylor expansion shows that $\Phi_\ell(\delta)\neq 0$ for all sufficiently small nonzero $\delta$.  Choosing such a $\delta$ with $\delta/\pi\in\mathbb Q$ gives a finite-order residual-circle element. Lemma~\ref{lem:donnelly-aps-normalization-factor} then gives the same separation for the corresponding Donnelly-APS $\eta$ values.
\end{proof}

\subsection{Finite-cyclic secondary pairing}
\label{subsec:stage5B4-finite-cyclic-secondary-pairing}

The preceding corollary gives a nonzero finite-order equivariant $\eta$/rho number.  The corresponding finite-cyclic class-function statement is the following.  It is an unconditional secondary $\eta$ statement at the level of finite-order equivariant $\eta$ values.

Fix odd $\ell\geq 5$.  Choose a sufficiently small rational $\delta/\pi$ such that $\Phi_\ell(\delta)\neq 0$, and let
\begin{equation}
g_\delta=e^{2i\delta}.
\end{equation}
Then $g_\delta$ has finite order.  Let $G_\delta=\langle g_\delta\rangle$ be the finite cyclic subgroup of the residual circle generated by $g_\delta$.

\begin{definition}[Normalized finite-cyclic equivariant rho class function]
For $h\in G_\delta$, define
\begin{equation}
\rho^{G_\delta,\text{norm}}_\ell(h)
=
\kappa_{\ell^2}^{-1}
\left[
\eta_h(D_{\mathcal Y^{\operatorname{res}}_{\ell,1}})
-
\eta_h(D_{\mathcal Y^{\operatorname{res}}_{\ell,2}})
\right].
\end{equation}
Thus $\rho^{G_\delta,\text{norm}}_\ell$ is the normalized complex-valued class function on $G_\delta$; the Donnelly-APS convention differs only by the common nonzero factor from Lemma~\ref{lem:donnelly-aps-normalization-factor}.
\end{definition}

\begin{theorem}[Finite-cyclic secondary separation]
\label{thm:finite-cyclic-secondary-separation}
For every odd $\ell\geq 5$, there exists a finite cyclic subgroup $G_\delta\subset S^1_{\text{res}}$ such that
\begin{equation}
\rho^{G_\delta,\text{norm}}_\ell\neq 0.
\end{equation}
More precisely,
\begin{equation}
\rho^{G_\delta,\text{norm}}_\ell(g_\delta)=\Phi_\ell(\delta)\neq 0.
\end{equation}
\end{theorem}

\begin{proof}
This follows directly from Corollary~\ref{cor:square-family-rho-separation-final}.  The chosen finite-order element $g_\delta$ satisfies $\Phi_\ell(\delta)\neq 0$, and by definition this value is $\rho^{G_\delta,\text{norm}}_\ell(g_\delta)$.
\end{proof}

\begin{corollary}[Nonzero finite-order evaluation pairing]
\label{cor:nonzero-finite-order-trace-pairing}
For every odd $\ell\geq 5$, there exists a finite cyclic subgroup $G_\delta\subset S^1_{\text{res}}$ and an element $g_\delta\in G_\delta$ such that
\begin{equation}
\operatorname{ev}_{g_\delta}(\rho^{G_\delta,\text{norm}}_\ell)\neq 0.
\end{equation}
\end{corollary}

\begin{proof}
The evaluation pairing is $\operatorname{ev}_{g_\delta}(F)=F(g_\delta)$.  Therefore
\begin{equation}
\operatorname{ev}_{g_\delta}(\rho^{G_\delta,\text{norm}}_\ell)=\rho^{G_\delta,\text{norm}}_\ell(g_\delta)=\Phi_\ell(\delta)\neq 0.
\end{equation}
\end{proof}

\section{Classical rho invariants and secondary refinements}
\label{sec:scope-limitations}

\subsection{Comparison with the classical lens-space rho setting}

The results above are residual-circle equivariant $\eta$ and rho-number separations.  The pair $L(25,4)$ and $L(25,9)$ also sits naturally inside the classical lens-space setting: 
\begin{equation}
4\cdot9\equiv6^2\pmod {25},
\end{equation}
while
\begin{equation}
9\not\equiv \pm4^{\pm1}\pmod {25}.
\end{equation}
Thus the pair is homotopy equivalent but not homeomorphic in the usual lens-space sense.  The result proved here is more specific: it exhibits a residual-circle $\eta$ germ for which the scalar $\eta$ value vanishes, the first derivative vanishes by symmetry, and the second derivative is nonzero:
\begin{equation}
\Phi(0)=0,
\qquad
\Phi'(0)=0,
\qquad
\Phi''(0)\neq 0.
\end{equation}
The resulting second-jet formula gives an infinite family of explicit residual-circle rho-number separations.  Equality of the ordinary $\eta$ value for the untwisted Dirac operator is not, by itself, equality of all finite-cyclic APS $\rho$-invariants; that would require comparing twists by representations of $\pi_1(L(p,q))$ and the relevant kernel normalizations.  We therefore keep the conclusion at the residual-circle level.

\subsection{Relation to analytic surgery and later secondary-class questions}

The residual circle used here is an additional geometric symmetry, whereas the standard analytic-surgery and higher-rho constructions are built from the fundamental group and its group $C^*$-algebra~\cite{HigsonRoeAnalyticKH,HigsonRoeSurgery,PiazzaSchick}.  A Higson-Roe structure-class refinement would therefore require extra work beyond the numerical $\eta$/rho statements proved here.  Related delocalized $\eta$ and rho-type refinements appear in finite and infinite fundamental-group settings, but those pairings are not used in the present argument~\cite{LottDelocalized}.

A possible continuation is to ask whether the second jet of the residual-circle $\eta$ germ admits a natural secondary $K$-theoretic interpretation, perhaps through an equivariant analytic-surgery group, equivariant $K$-homology, or eta-form/transgression formalism.  A possible singular-foliation extension would require replacing the present free finite-quotient model with a holonomy groupoid framework, in the sense of Androulidakis-Skandalis~\cite{AndroulidakisSkandalis}.  This is not needed for the present eta-separation theorem and is left for separate work.

\FloatBarrier

\appendix

\section{Full first-useful spin-Fourier residue table}
\label{app:first-useful-spin-fourier-table}

\begin{table}[!htbp]
\centering
\small
\setlength{\tabcolsep}{4pt}
\renewcommand{\arraystretch}{1.12}
\begin{tabular}{c|c|c|c|c}
$p$ & $q$ & first useful $k$ & $\chi_k^+(p,q;u)$ & $\chi_k^-(p,q;u)$ \\
\hline
$5$ & $2$ & $1$ & $A_1$ & $A_3$ \\
$5$ & $3$ & $1$ & $A_3$ & $A_1$ \\
\hline
$7$ & $2$ & $2$ & $A_1$ & $A_5$ \\
$7$ & $3$ & $2$ & $A_1$ & $A_3$ \\
$7$ & $4$ & $2$ & $A_3$ & $A_1$ \\
$7$ & $5$ & $2$ & $A_5$ & $A_1$ \\
\hline
$9$ & $2$ & $2$ & $A_1$ & $A_3$ \\
$9$ & $4$ & $2$ & $A_3$ & $A_5$ \\
$9$ & $5$ & $2$ & $A_5$ & $A_3$ \\
$9$ & $7$ & $2$ & $A_3$ & $A_1$ \\
\hline
$11$ & $2$ & $3$ & $A_1$ & $A_5$ \\
$11$ & $3$ & $3$ & $A_3$ & $A_1$ \\
$11$ & $4$ & $3$ & $A_1$ & $A_7$ \\
$11$ & $5$ & $3$ & $A_3$ & $A_5$ \\
$11$ & $6$ & $3$ & $A_5$ & $A_3$ \\
$11$ & $7$ & $3$ & $A_7$ & $A_1$ \\
$11$ & $8$ & $3$ & $A_1$ & $A_3$ \\
$11$ & $9$ & $3$ & $A_5$ & $A_1$ \\
\hline
$13$ & $2$ & $3$ & $A_1$ & $A_3$ \\
$13$ & $3$ & $4$ & $A_1$ & $A_3$ \\
$13$ & $4$ & $4$ & $A_1$ & $A_9$ \\
$13$ & $5$ & $2$ & $A_5$ & $A_1$ \\
$13$ & $6$ & $3$ & $A_5$ & $A_7$ \\
$13$ & $7$ & $3$ & $A_7$ & $A_5$ \\
$13$ & $8$ & $2$ & $A_1$ & $A_5$ \\
$13$ & $9$ & $4$ & $A_9$ & $A_1$ \\
$13$ & $10$ & $4$ & $A_3$ & $A_1$ \\
$13$ & $11$ & $3$ & $A_3$ & $A_1$
\end{tabular}
\caption{Full first-useful spin-Fourier residue table for odd $p\leq13$.  Here $A_n=A_n(u)=u^n+u^{-n}$.}
\label{tab:first-useful-spin-fourier}
\end{table}

This appendix gives the full small-range table used only as a bookkeeping check for the affine congruence and parity-shell formula in~\autoref{subsec:stage-1C-first-table}.  The proof of the main residual-circle $\eta$ separation does not depend on this table.

\FloatBarrier

\section{Exact verification of the \texorpdfstring{$L(25,4)$ versus $L(25,9)$}{L(25,4) versus L(25,9)} equivariant $\eta$ second derivative}
\label{app:L25-eta-second-derivative}

This appendix gives the finite algebraic check used in~\autoref{subsec:L254-L259-equivariant-eta}.  Let $x_j=\pi j/25$ and $C_j=\csc(29x_j)-\csc(9x_j)$.  The claimed second derivative is the finite trigonometric identity
\begin{equation}
2\sum_{j=1}^{12}
C_j
\left[
\csc(x_j)\cot^2(x_j)+\csc^3(x_j)
\right]
=
-6080.
\end{equation}
This can be checked exactly in the cyclotomic field generated by
\begin{equation}
\zeta=e^{\pi i/25}.
\end{equation}
Using
\begin{equation}
\sin(mx_j)=\frac{\zeta^{mj}-\zeta^{-mj}}{2i},
\qquad
\cos(x_j)=\frac{\zeta^{j}+\zeta^{-j}}{2},
\end{equation}
each summand becomes a rational function of $\zeta$.  After bringing the sum to a common denominator and reducing by the cyclotomic equation $\Phi_{50}(\zeta)=0$, the numerator of
\begin{equation}
2\sum_{j=1}^{12}
C_j
\left[
\csc(x_j)\cot^2(x_j)+\csc^3(x_j)
\right]
+6080
\end{equation}
reduces to zero.  Hence the identity is exact.

\section{Symbolic verification of the block moment identities}
\label{app:block-moment-verification}

For reproducibility, this appendix gives an exact symbolic check of the block moment identities used in Lemma~\ref{lem:block-moment-identities}.  The proof of the lemma reduces the moments to ordinary power sums; the symbolic check below gives an independent exact audit of the resulting closed forms.  The symbolic reductions were performed in Mathematica using exact integer and rational arithmetic; no floating-point approximation is used.

Let $N$ be an odd integer and let
\begin{equation}
P=N^2.
\end{equation}
For $0\leq n,t\leq N-1$, define
\begin{equation}
\varepsilon_n=
\begin{cases}
1,& n=0,\\
0,& n\neq 0,
\end{cases}
\qquad
c_n=[n-1]_N,
\end{equation}
and
\begin{equation}
\beta_1(n)=n+\frac{N+1}{2}+\varepsilon_n,
\qquad
\beta_2(n)=2n+1+\varepsilon_n.
\end{equation}
Then
\begin{equation}
R_i(t,n)=c_n+N[t-\beta_i(n)]_N.
\end{equation}
The moments are
\begin{equation}
M_k(N)
=
\sum_{n=0}^{N-1}
\sum_{t=0}^{N-1}
(Nt+n)\left(R_1(t,n)^k-R_2(t,n)^k\right).
\end{equation}

The exact symbolic audit gives the following closed forms.  First,
\begin{equation}
M_1(N)=\frac{N^3(N-1)}{2}.
\end{equation}
Second,
\begin{equation}
M_2(N)=\frac{N^3(N-1)(N^2-2)}{2}.
\end{equation}
Third,
\begin{equation}
M_3(N)=
-\frac{N^3(N-1)
\bigl(11N^5-460N^4+70N^3+1420N^2-81N-1440\bigr)}{960}.
\end{equation}
Substituting these three identities into the assembly identity used in Lemma~\ref{lem:block-moment-identities} gives
\begin{equation}
D_3(N)=
-\frac{N^2(N-1)^2(N+1)(11N^2+20N+81)}{5760}.
\end{equation}
Since $D_1(N)=0$, the normalized second derivative is
\begin{equation}
C_\ell
=
32D_3(\ell)
=
-\frac{\ell^2(\ell-1)^2(\ell+1)(11\ell^2+20\ell+81)}{180}.
\end{equation}
This agrees with the direct trace computation from Lemma~\ref{lem:D3-trace-to-convolution} and with the finite trigonometric check in the case $\ell=5$.

\bibliographystyle{utphys}
\bibliography{ref}

\end{document}